\documentclass[twoside,a4paper,10pt]{amsart}
\usepackage{amssymb}
\usepackage{amsmath}
\usepackage{graphicx}
\usepackage[ruled,vlined]{algorithm2e}
\usepackage{xcolor}
\usepackage{cancel}

\newtheorem{theorem}{Theorem}
\newtheorem{definition}{Definition}

\newtheorem{remark}[theorem]{Remark}

\newcommand{\inte }{{\rm int}\,}

\newcommand{\amin}{a_{\mathrm{min}}}
\newcommand{\amax}{a_{\mathrm{max}}}
\newcommand{\xmin}{x_{\mathrm{min}}}
\newcommand{\xmax}{x_{\mathrm{max}}}
\newcommand{\zmin}{z_{\mathrm{min}}}
\newcommand{\zmax}{z_{\mathrm{max}}}

\newcommand{\cover}[1]{\stackrel{#1}{\Longrightarrow}}

\begin{document}
\title[A mechanism for growth of topological entropy]{A mechanism for growth of topological entropy and global changes of the shape of chaotic attractors}

\author{Daniel Wilczak}
\email{Daniel.Wilczak@uj.edu.pl}
\address{Faculty of Mathematics and Computer Science.
	Jagiellonian University. \L ojasiewicza 6, 30-348 Krak\'ow, Poland.
}

\author{Sergio Serrano}
\email{sserrano@unizar.es}
\address{Departamento de Matem\'atica Aplicada and IUMA.
	University of Zaragoza. E-50009. Spain.\\ CODY. University of Zaragoza.
	E-50009. Spain.}

\author{Roberto Barrio}
\email{rbarrio@unizar.es}
\address{Departamento de Matem\'atica Aplicada and IUMA.
	University of Zaragoza. E-50009. Spain.\\ CODY. University of Zaragoza.
	E-50009. Spain.}

\date{\today}
\begin{abstract}
The theoretical and numerical understanding of the key concept of topological entropy is an important problem in dynamical systems. Most studies have been carried out on maps (discrete-time systems).
We analyse a scenario of global changes of the structure of an attractor in continuous-time systems leading to an unbounded growth of the topological entropy of the underlying dynamical system. As an example, we consider the classical R\"ossler system. We show that for an explicit range of parameters a chaotic attractor exists. We also prove the existence of a sequence of bifurcations leading to the growth of the topological entropy. The proofs are computer-aided.

\end{abstract}

\maketitle

\section{Introduction.}

The concept of topological entropy was first introduced by Adler, Konheim and McAndrew~\cite{Ent}, and it is one of the most important topological invariants in dynamical systems theory because it is one of the possible ways to measure the complexity of dynamics.
It expresses the exponential growth rate of the number of distinguishable orbits the system can create under iteration and it is  widely used in the theoretical description of the transition to chaos. The changes in topological entropy with the parameter of the system indicate bifurcations, which affect global orbit structure. As  Milnor~\cite{Milnor02} points out, it is natural to ask whether topological entropy can be calculated efficiently. Although in most cases it is impossible to compute the entropy exactly~\cite{GHRS20,Milnor02}, in many situations it can be bounded from below, for instance by proving (semi) conjugacy to a shift dynamics. Obviously, there are many open questions in the study of dynamical systems, and the behaviour and proof of the values of the topological entropy on different systems is one of them.

The aim of this paper is to present an algorithmic approach to prove the existence of global changes of the shape of attractors in continuous-time systems when a parameter of the system is varying. As a paradigmatic example we consider the R\"ossler system~\cite{ROSSLER1976397}
\begin{equation}\label{eq:rossler}
	\dot x = -(y+z),\qquad \dot y = x+ay,\qquad \dot z = b + z(x-c).
\end{equation}
In the literature this model has been extensively studied~\cite{BBS09,BBS14,BBSS11,BBS12,ROSSLER1976397,Z}, showing different types of chaotic attractors and changes in the first return maps.

We will prove that for a range of parameter values the attractor exists and it undergoes bifurcations leading to global changes of its shape. Although the exact value of the entropy for the R\"ossler system is unlikely to compute, we will rigorously estimate the entropy from below in different subintervals of systems' parameter and we will show that this lower bound is growing with the parameter of the system. We will also prove that there is a sequence of saddle-node bifurcations, which give rise to semiconjugacy of certain Poincar\'e map to the Bernoulli shift on $2$ up to $13$ symbols, depending on parameter range. 

The main problem is that to provide an analytical proof of these changes is simply not possible, but for very simple models. Therefore, in order to consider a classical and seminal model, we will use powerful machinery of  so-called \emph{validated numerics} \cite{Neumaier_1991,Moore,TuckerBook}.

In the last decades efficient algorithms for validated integration of finite dimensional ODEs have been proposed \cite{NedialkovJackson1998,NedialkovJacksonCorliss1999,NedialkovJacksonPryce2001,Nedialkov2006,RauhBrillGunter2009,capd,C1Lohner,C1HO,CnLohner}. Even if the explicit solutions to an ODE cannot be computed exactly, one can often computed bounds on the flow \cite{capd} or Poincar\'e maps \cite{poincare} and their derivatives. Using these bounds we can check if the solutions satisfy certain inequalities and thus extract partial information of the dynamics. The field of computer-assisted techniques \cite{CAPINSKI2023106998} is currently a very active area~\cite{CapinskiFleurantinJames2020,BARTHA2015339,Capinski2012,CyrankaWanner,GalanteKaloshin,cadiot2024rigorous,DLT,BARRIO201280,WSB16} as it provides rigorous numerical techniques to obtain proofs for different phenomena observed just numerically including the solution to the Smale's 14th problem \cite{Tucker2002} about the existence and non-uniform hyperbolicity of the Lorenz attractor.

The paper is organized as follows. In Section~\ref{sec:sin_model} we present a one-dimensional scenario that leads to global changes of an invariant set. In Section~\ref{sec:numbif} we present result of biparametric numerical analysis of the R\"ossler system (\ref{eq:rossler}). In Section~\ref{sec:rossler} the main results about the existence of attractor in the R\"ossler system and changes of its structure are given. These are Theorem~\ref{thm:attractor}, Theorem~\ref{thm:template}, Theorem~\ref{thm:bifurcation} and Theorem~\ref{thm:entropy}. In Section~\ref{sec:CAP} we give computer-assisted proofs of the above main results.

\section{Toy model -- the sine map}\label{sec:sin_model}
\begin{figure}[htb]
	\centerline{\includegraphics[width=\textwidth]{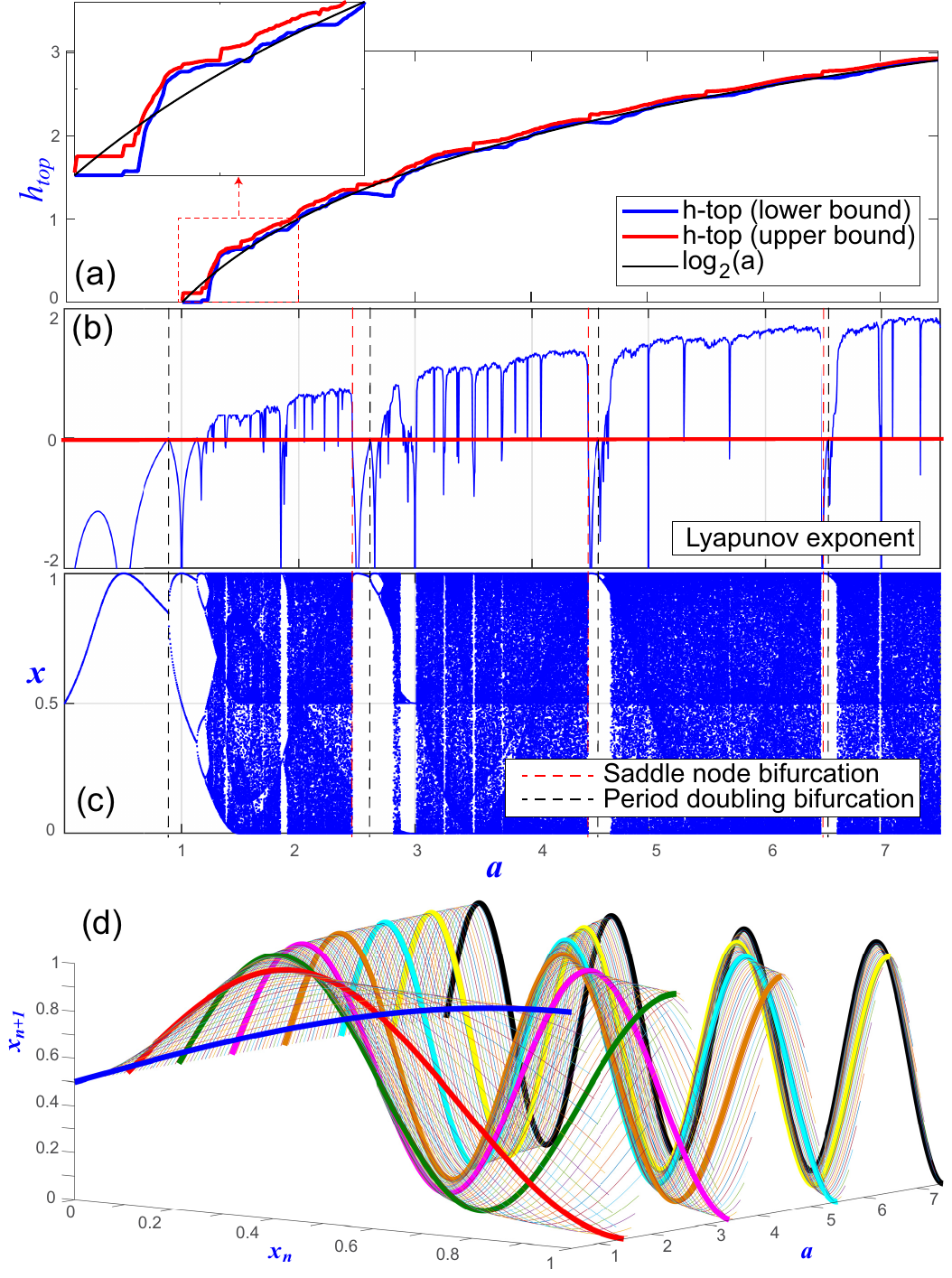}}
	\caption{(a): Lower and upper bounds  for the values of the topological entropy on the range $a \in [0, 7.5]$ and the value of $\log_2(a)$. (b) Lyapunov exponent of the map (\ref{eq:SinModel}). (c): bifurcation diagram. (d): 3D plot of the function map (\ref{eq:SinModel}) depending on the parameter $a$. \label{fig:cos_map}}
\end{figure}

In order to present geometry of the mechanism that leads to growth of the entropy, let us consider the following one-dimensional toy model
\begin{equation}\label{eq:SinModel}
	f_a(x) = \frac{1}{2}\left(1+\sin(a\pi x)\right), \qquad \qquad a\in\mathbb R.
\end{equation}
Clearly the interval $I=[0,1]$ is a forward invariant set for any parameter value $a\in\mathbb R$.
In Figure~\ref{fig:cos_map} we present on the top (plot (a)) lower and upper bounds for the values of the topological entropy (calculated using the algorithm described in \cite{GoraBoyarsky}) on the range $a \in [0,\, 7.5]$. The upper bound is computed as $\frac{1}{k}\log_2 c_k$ for some large $k$, where $c_k$ is the number of monotone slopes of $f_a^k$. We observe that for larger values of the parameter $a$ the entropy is growing like $\log_2a$. On plot (b) we show the Lyapunov exponent on the parametric interval $a\in[0, \, 7.5]$. The plot shows how there are some regions with chaotic behavior that repeat more or less at distance 2 in the parameter $a$. On the plot (c) we present the bifurcation diagram showing the dynamics. We observe clearly different saddle-node bifurcations, as before more or less at distance 2 in the parameter $a$, and period-doubling cascades leading to chaos as shown by the Lyapunov exponent. And on the bottom figure (d) we observe the function map (\ref{eq:SinModel}) changing the parameter $a$. From the figure we see how there are changes in the number of monotone branches, passing from unimodal maps to multiple modal ones. The thick lines denote the limit maps between $n$-modal and $n+1$-modal maps.

If we focus on the first changes of the map of Eqn.~(\ref{eq:SinModel}), we observe on the bottom plot that for the initial parameter value $a=1/2$, the function is strictly monotone and thus the only limit sets are fixed points. At $a\approx 2.45855$ a saddle-node bifurcation occurs creating a pair of stable and unstable fixed points -- see Fig.~\ref{fig:CosModel}. For $a\approx 2.6175$ the stable fixed point losses stability through the period doubling bifurcation, which is an onset of chaotic dynamics via the well known cascade of period doubling bifurcations.
\begin{figure}[htbp]
	\centerline{\includegraphics[width=\textwidth]{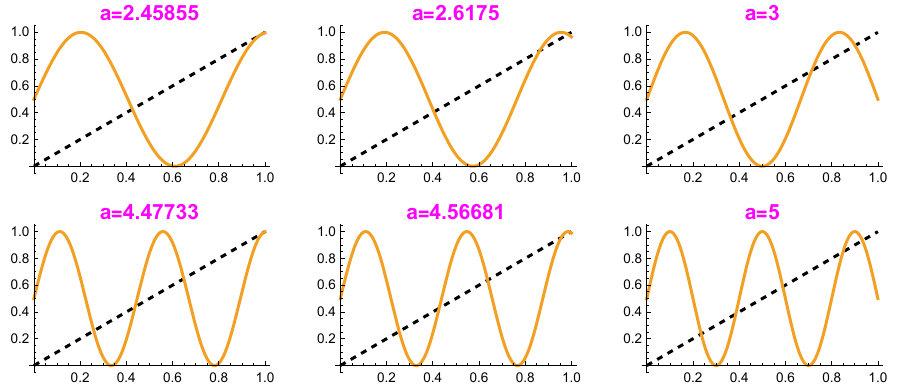}}
	\caption{Plot of $f_a(x)$ for different parameter values. A saddle-node bifurcation creates a pair of fixed points (left column). The stable fixed point losses stability via period doubling bifurcation (middle column). Further growth of parameter leads to creation of new symbol for conjugacy to symbolic dynamics (right column).\label{fig:CosModel}}
\end{figure}
Finally, for $a=3$ the interval $[0,1]$ can be split into subintervals $N_1=[0,\frac{1}{6}]$, $N_2=[\frac{1}{6},\frac{1}{2}]$, $N_3=[\frac{1}{2},\frac{5}{6}]$ and $N_4=[\frac{5}{6},1]$ in which the function is monotone. Every point $x\in N_1$ is mapped to either $N_3$ or $N_4$. The images of both $N_2$ and $N_3$ cover the range  $[0,1]$. Finally, the points $x\in N_4$ can be mapped to $N_3$ or $N_4$. Observe also, that for almost every $x\in [0,1]$, excluding countable set of points $\mathcal E= \bigcup_{i\in\mathbb N} f_{a=3}^{-i}\left(\{0,\frac{1}{6},\frac{1}{2},\frac{5}{6},1\}\right)$, the trajectory $\{f_{a=3}^i(x)\}_{i\in\mathbb N}$ visits the interiors of $N_i$'s only, which are pairwise disjoint sets. For every $x\in\mathcal I:=[0,1]\setminus \mathcal E $ the trajectory can be encoded as an infinite path on the directed graph shown in Fig.~\ref{fig:SampleGraph}, that is a sequence of vertices $(c_i)_{i\in\mathbb N}\in\{1,2,3,4\}^{\mathbb N}$, such that $f^i_{a=3}(x)\in \mathrm{int}N_{c_i}$ for $i\in\mathbb N$. This gives rise to semiconjugacy between $f_{a=3}|_{\mathcal I}$ and so-called symbolic dynamics on four symbols. This notion will be introduced  in Section~\ref{sec:rossler}.
\begin{figure}[htbp]
	\centerline{\includegraphics[width=.5\textwidth]{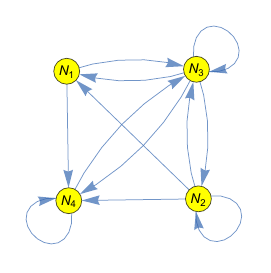}}
	\caption{Graph of symbolic dynamics of the sine model (\ref{eq:SinModel}) for $a=3$.\label{fig:SampleGraph}}
\end{figure}

Such scenario repeats when the parameter $a$ is growing. We observe a sequence of saddle node bifurcation, which create a new pair of stable-unstable fixed points. In fact, when $a=(2n+1)/2$ for $n \in \mathbb{N}$ a new extremum (as it can be seen in Fig.~\ref{fig:CosModel} (bottom)) of the map appear and so a new symbol. In consequence, the topological entropy of $f_a$ grows to infinity with $a\to\infty$.

We have briefly seen that in a simple case of the sine map, interesting changes in the topological entropy happen. For instance, in the celebrated article~\cite{Milnor88} it was proved that the entropy function is monotonically increasing for the quadratic (logistic) family. So, a clear extension and an interesting open question is to see  what happens for continuous systems. On that direction, in Section~\ref{sec:numbif} we present results of a biparametric numerical study of the R\"ossler system (\ref{eq:rossler}) exhibiting a similar mechanism of saddle-node bifurcations that may lead to growth of the entropy of the system. Then, in Section~\ref{sec:rossler}, we present an algorithmic approach to provide a computer-assisted proof of the existence the R\"ossler attractor for an explicit range of parameter values, validation of the existence of saddle-node bifurcations and for computation of a lower bound of the topological entropy.

\section{Numerical study of the R\"ossler system.}

Unlike the sine model (\ref{eq:SinModel}), R\"ossler system (\ref{eq:rossler}) depends on three parameters. In this paper we fix the value of $b=0.2$, similar results could be obtained with other values (see \cite{BBS09} for a parametric study of the system). Figure \ref{fig:ExpoConBif} shows a biparametric plate in which the values of the parameters $a$ and $c$ vary. The plate shows the values of the first two Lyapunov exponents \cite{W85}. The colours (from green to red) represent the chaotic region, identified by the fact that the maximum Lyapunov exponent is greater than 0. A grey gradient represents the regular region, where the maximum exponent cancels out and the second Lyapunov exponent is represented. Black means that the second exponent is close to zero and light grey means that the periodic attractor is more stable. The white region on the right (with values of $a$ close to $0.37$) indicates that the dominant dynamics is the escape dynamics (see \cite{BBS14} for a detailed explanation of the unbounded dynamics of the system). Superimposed on this plate are several bifurcation curves obtained with the well-known continuation software AUTO \cite{AUTO1, AUTO2}. In blue for saddle nodes and in red for period doubling. The illustration is not exhaustive, but is intended to show the huge number of bifurcations in the region shown. This results in a mixture of regions with different dynamics. To study the evolution of the model dynamics in more detail, we select a segment (with $c=15$, marked in white and dashed grey) that crosses many of these regions and almost reaches the region of unbounded dynamics.

\label{sec:numbif}
\begin{figure}[htbp]
\centerline{\includegraphics[width=.8\textwidth]{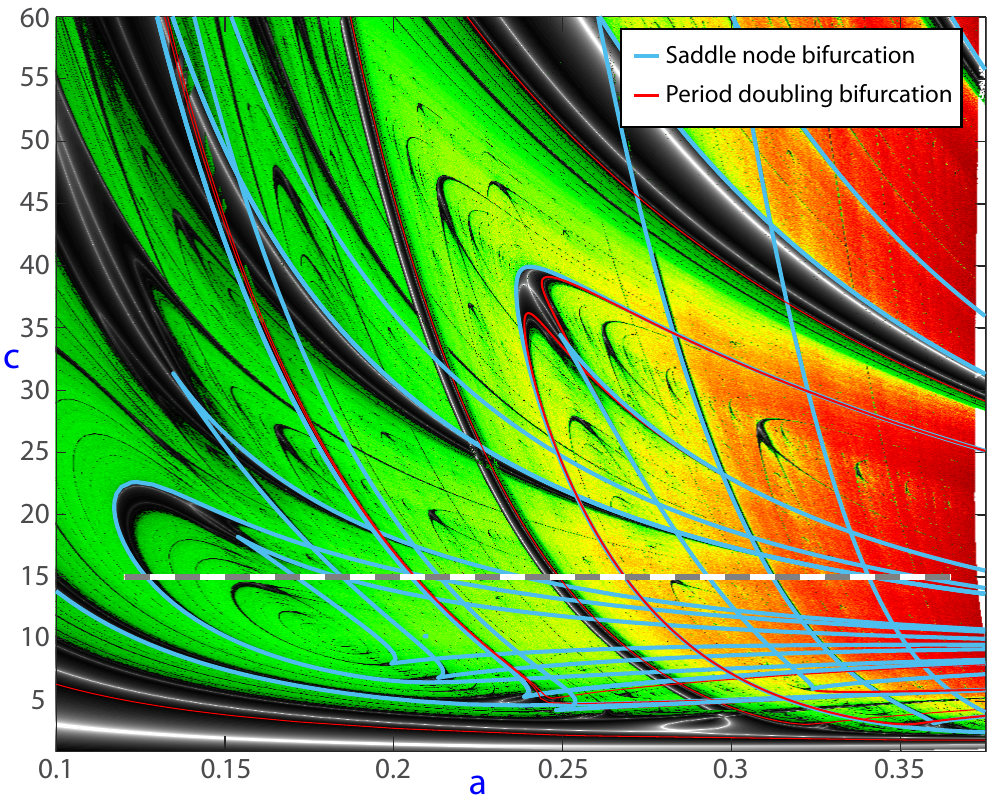}}
\caption{Biparametric plot of the largest Lyapunov exponents of (\ref{eq:rossler}) with $b=0.2$.
	Color (from green to red) represents chaotic regimes shown by the maximum Lyapunov exponent. In black and
	white, regular dynamics, first exponent is null and the second is
	represented. Blue curves mark saddle node bifurcations, while red curves
	indicate period doubling bifurcations. The white and grey dashed segment marks the line (with $c=15$) that we will study in more detail in the rest of the paper. \label{fig:ExpoConBif}}
\end{figure}

Given a suitable Poincar\'e section and $p_n$ the successive intersections of an orbit of the system with the previous Poincar\'e section, we can define the first return map of the orbit as FRM$(x_n) = x_{n+1}$ (where $x_n$ are the values of a selected coordinate at the successive points $p_n$). Like the R\"ossler system (\ref{eq:rossler}), many dynamical systems modelling problems of different nature are strongly dissipative \cite{BW83, LDM95, BLBD98, UM10, BBSS11, BBS12, SMB21}. Their dynamics is characterised by the fact that the contraction of their flow along the stable manifold of their equilibrium points is much larger than the expansion along the unstable manifold of their equilibrium points. For such strongly dissipative systems, the FRM allows to obtain a qualitative description of the topology of invariant chaotic sets \cite{LDM95, G98, GL02}.

\begin{figure}[htbp]
  \centerline{\includegraphics[width=.9\textwidth]{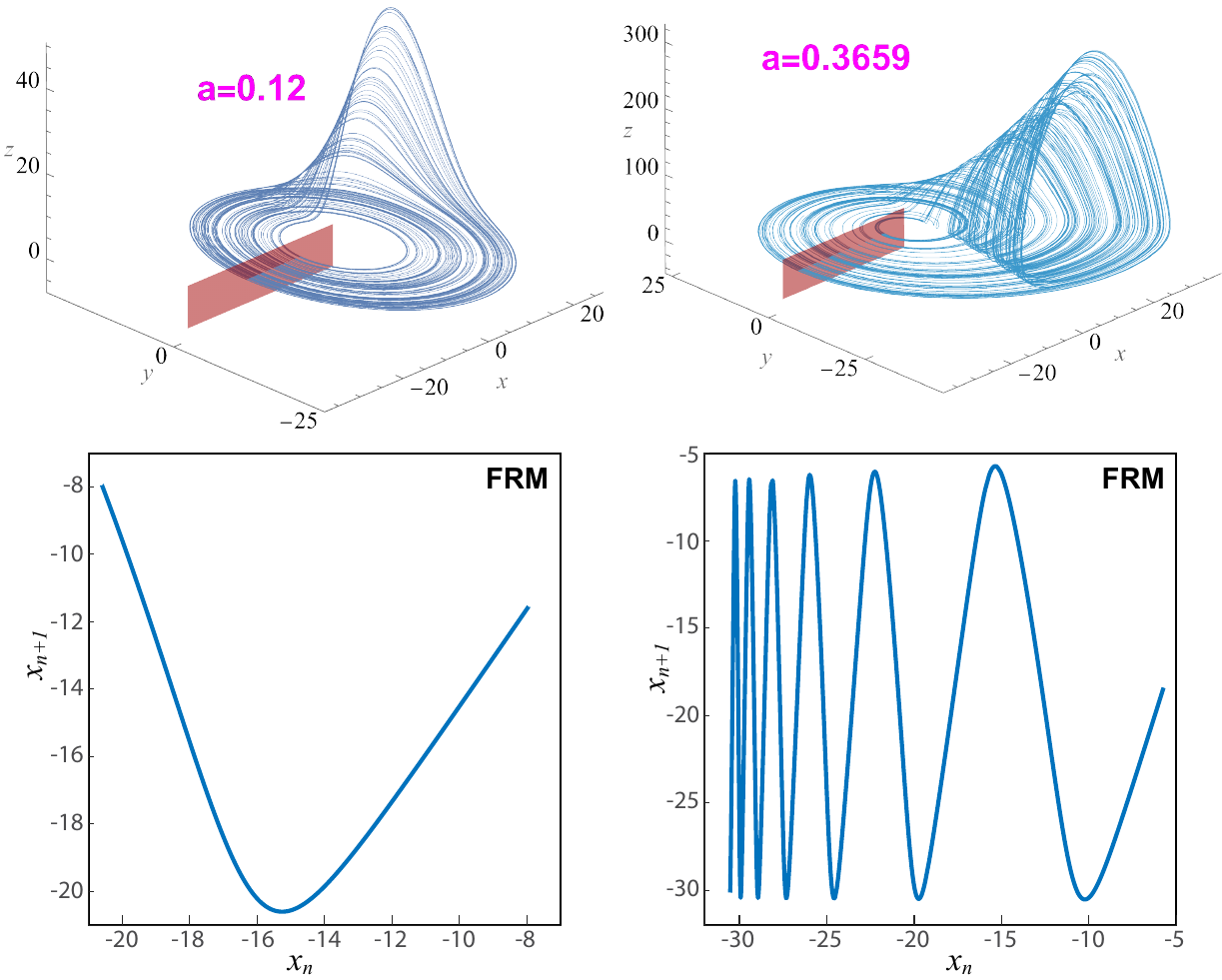}}
 \caption{Top: Poincar\'e section and a typical trajectory of (\ref{eq:rossler}) for $a=\amin=0.12$ (left) and $a=\amax=0.3659$ (right). Bottom: The corresponding FRM.\label{fig:AttractorAndSection}}
\end{figure}

For our analysis, we define Poincar\'e section
\begin{equation}\label{eq:PoincareSection}
	\Pi := \{(x,y,z)\in \mathbb R^3 : y=0 \wedge \dot y=x<0 \}.
\end{equation}
In Figure \ref{fig:AttractorAndSection} we show the chaotic attractor, the Poincar\'e section and the corresponding FRM for the two values of parameter $a$ ($0.12$ and $0.3659$) located at the two ends of the segment marked in Fig.~\ref{fig:ExpoConBif}. As can be seen, the structure of the second attractor (funnel type) is much more complex than that of the first (spiral type). This increase in complexity is reflected in a higher number of branches of the FRM.  If we obtain and plot the FRM of the invariant chaotic set on a sufficiently fine mesh of the selected interval, we can observe that this increase in complexity of the chaotic attractor occurs gradually. This is shown in Fig.~ \ref{fig:frm3d}. In this figure we have highlighted in a thicker line some FRMs that are approximately at the values of $a$ where the change from $n$ to $n+1$ branches for $n$ from 2 to 8 (the following changes appear closer together and it would be necessary to use a finer mesh to mark them with sufficient precision). Note that the calculations use the chaotic set both when it is an attractor and when it is a saddle. If the chaotic set is an attractor, then a good approximation of it, and therefore of its FRM, is easily obtained. However, when a stable periodic orbit coexists with the chaotic saddle, it is not easy to obtain. In this work, we use the Sprinkle method \cite{KG85}, since the transition time near the chaotic set is large and this allows us to obtain a good approximation of it.

\begin{figure}[htbp]
  \centerline{\includegraphics[width=.9\textwidth]{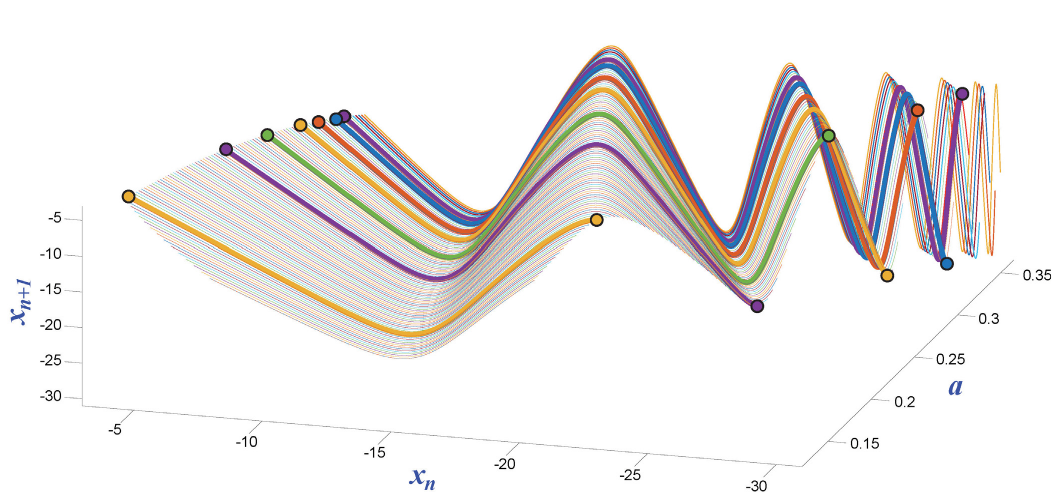}}
 \caption{FRMs for the existing invariant chaotic sets along the selected segment in Figure \ref{fig:ExpoConBif}. The highlighted FRMs roughly indicate the increase in an additional branch.\label{fig:frm3d}}
\end{figure}

We can study this evolution of the dynamics along the selected segment (with $c=15$) using different techniques. Figure \ref{fig:c15} shows at the top a lower bound on the entropy of the system (see subsection \ref{subsec:entropy} for more details). In the middle part we can see the first two Lyapunov exponents. As mentioned above, the first positive exponent indicates that the attractor is chaotic. Both indicators show the trend of increasing complexity of the chaotic attractor as we move up the segment by increasing the $a$ value. The second Lyapunov exponent gives us information in the regions where the attractor is regular. So if this exponent has very negative values, the periodic orbit is more stable. On the other hand, if it rises to zero at one point and then falls again, it indicates a period-doubling bifurcation. If the first exponent drops vertically from positive values to zero and the second exponent goes from 0 to negative values, we have a saddle-node bifurcation. Some of these bifurcations are marked with dashed green and purple segments, respectively. The bottom of the figure shows the bifurcation diagram obtained with the selected Poincaré section. In this bifurcation diagram, the regular and chaotic regions detected by the Lyapunov exponents are easily observed. We also mark with black dotted segments the value of $a$ at which a new branch appears in the FRM. As we can see, transitions from an odd number of branches to an even number of branches always occur in a regular window. In these windows we have marked the saddle-node bifurcations that give rise to them, as well as the two families of periodic orbits (stable in blue and unstable in red) that arise from them.

\begin{figure}[htbp]
  \centerline{\includegraphics[width=.9\textwidth]{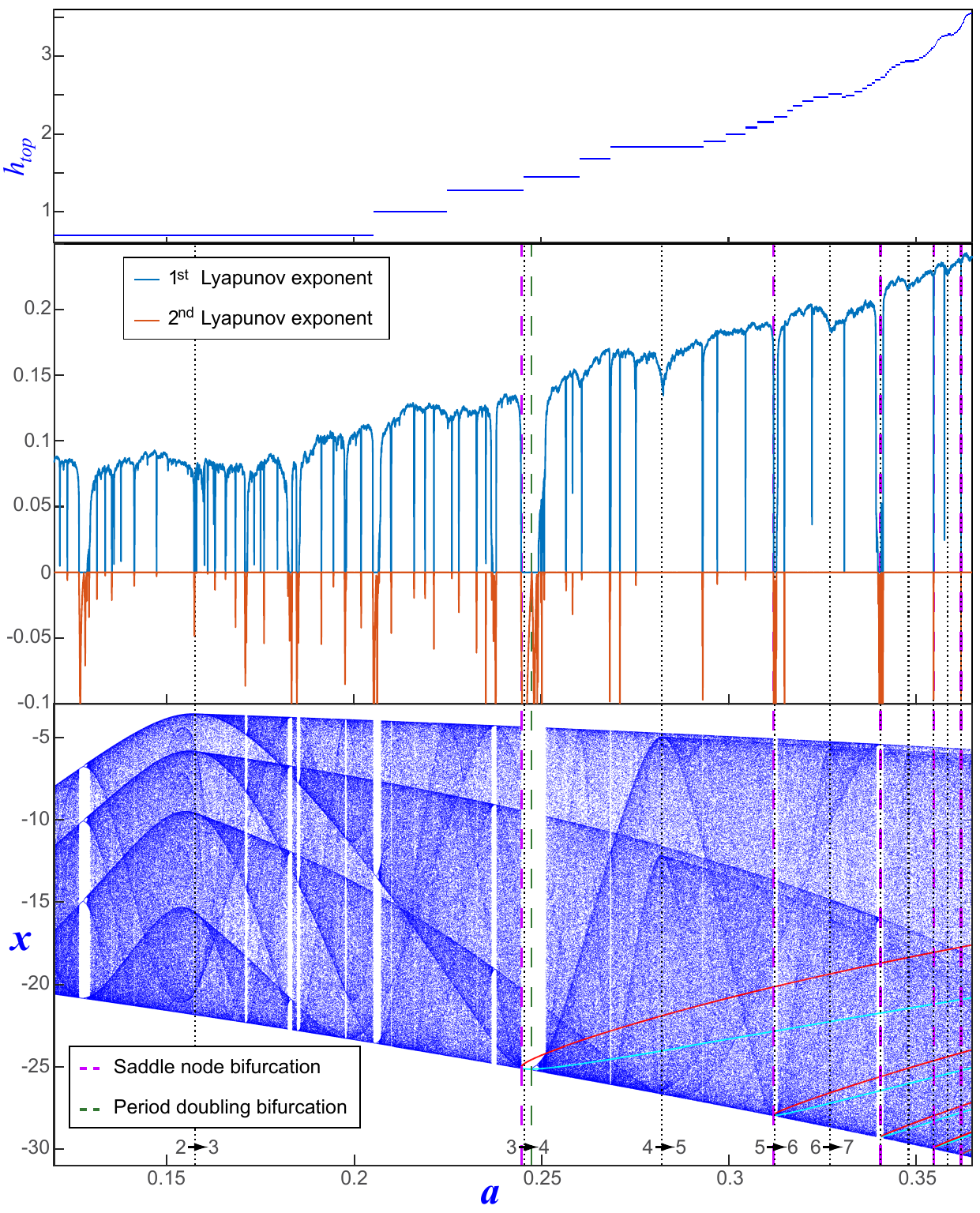}}
 \caption{Analysis of the evolution of the dynamics of the R\"ossler model along the selected segment in Figure \ref{fig:ExpoConBif} using different techniques. Top: lower bound for the topological entropy (see subsection \ref{subsec:entropy}). Middle: Lyapunov exponents. Bottom: Bifurcation diagram representing the $x$-coordinate of the intersection points of the attractor with the Poincar\'e section $\Pi$. The dashed vertical segments mark some saddle-node (in purple) and period doubling (in green) bifurcations.  Blue and red curves show stable and unstable families of periodic orbits born from previous saddle-node bifurcations.  The black dotted segments mark the approximate value by which the FRM increases by one branch. Above the $x$-axis are the number of branches to the left and right at the first transitions.\label{fig:c15}}
\end{figure}


\section{The main results.}
\label{sec:rossler}
During this section we fix the parameter values $b=0.2$ and $c=15$ and we will study the dynamics of the R\"ossler system (\ref{eq:rossler}) in the range of parameter values
\begin{equation}\label{eq:arange}
	\mathcal A=[\amin,\amax]:=[0.12,0.3659].
\end{equation}

On the Poincar\'e section $\Pi$ defined by (\ref{eq:PoincareSection}) we will use $(x,z)$ coordinates. For a parameter value $a$ we define the following Poincar\'e map
\begin{equation}\label{eq:PoincareMap}
	P_a:\Pi\to\Pi.
\end{equation}

In this section we present several results about Poincar\'e map $P_a$. The proofs of all theorems are computer-assisted and the details will be presented in Section~\ref{sec:CAP}.

\begin{theorem}\label{thm:attractor}
For $a\in \mathcal A$ the Poincar\'e map $P_a$ is well defined and smooth on the set
\begin{equation}\label{eq:defTrappingRegion}
	\mathcal T:=[\xmin,\xmax]\times[\zmin,\zmax] = [-30.53,-3]\times[0.004,0.011]\subset \Pi
\end{equation}
and $P_a(\mathcal T)\subset \mathrm{int}\mathcal T$.
\end{theorem}
Thus, the set $\mathcal T$ is a trapping region for the Poincar\'e map $P_a$ for $a\in\mathcal A$ which contains nonempty, compact and connected maximal invariant set
$$
	\mathcal I_a := \bigcap_{n>0}P^n_a(\mathcal T).
$$

\subsection{Change of the shape of attractor.}
Section~\ref{sec:numbif} (see also Fig.~\ref{fig:AttractorAndSection}, Fig.~\ref{fig:frm3d} and Fig.~\ref{fig:c15}) strongly indicates that structure of the invariant set $\mathcal I_a$ is changing with the parameter $a\in[\amin,\amax]$. In this section we introduce some computable characteristic of the attractor and we will prove that it indeed changes when the parameter $a$ is varying.

For fixed parameter $a\in[\amin,\amax]$ and fixed $z\in[\zmin,\zmax]$ we define the following function
\begin{equation}\label{eq:faz}
f_{a,z}(x):= \pi_x P_a(x,z),
\end{equation}	
where $\pi_x$ is the projection onto $x$-coordinate. From Theorem~\ref{thm:attractor} it follows that for all $a\in[\amin,\amax], z\in[\zmin,\zmax]$ the function $f_{a,z}$ is smooth on $X:=[\xmin,\xmax]$ and its range is also in $X$.

\begin{definition}\label{def:relEx}
	Let $X$ be a closed interval and let $f:X\to X$ be continuous. A point $x_*\in\inte X$ is called a \emph{relevant extremum} of $f$ if $x_*$ is a proper local extremum and
	$$\min_{x\in X} f(x)\leq x_*\leq \max_{x\in X} f(x).$$
	The cardinality of the set of relevant extrema of $f$ is denoted by $\mathrm{relEx}(f)$.
\end{definition}
This definition is motivated by the following observation regarding the sine model (\ref{eq:SinModel}). Clearly any interval $[0,N]$, $N\geq1$ is forward invariant for $f_a$. However, all local extrema located out of the range of $f_a$, for example satisfying $x>1$, do not affect the structure of the invariant set in $[0,1]$. Thus, only the extrema that belong to the range $f_a([0,N])$ are relevant for the dynamics on the invariant set and their number provides some characteristic of the invariant set. That is, we are eliminating the transient dynamics created when the initial point is out of the interval of the invariant set. In most cases this transient is just one iteration of the map.

\begin{theorem}\label{thm:template}
	For parameter values $a$ listed in (\ref{eq:templateParameters}) the number of relevant extrema of $f_{a,z}:X\to X$ defined by (\ref{eq:faz}) does not depend on $z\in[\zmin,\zmax]$ and it is equal to
	
	\begin{equation}\label{eq:templateParameters}\scriptsize
	\begin{array}{c|c|c|c|c|c|c|c|c|c|c|c|c}
		a & \amin & 0.2 & 0.26 & 0.3 & 0.32 & 0.333 & 0.345 & 0.35 & 0.356 & 0.36 & 3.62 & \amax\\ \hline
		\mathrm{relEx}(f_{a,z}) & 1 & 2 & 3& 4& 5& 6& 7& 8& 9 & 10 & 11 & 12
	\end{array}
\end{equation}
\end{theorem}
The plot of $f_{a,z}$ for $z=\frac{1}{2}(\zmin+\zmax)$ and for $a\in\{\amin,\amax\}$ is shown in Fig.~\ref{fig:RelevantExtrema}. Theorem~\ref{thm:template} implies that between each pair of subsequent parameter values listed in (\ref{eq:templateParameters}) a global change in the structure of $\mathcal I_a$ occurs.

Indeed, we have the following theorem.
\begin{theorem}
The number of relevant extrema $\mathrm{relEx}$ is an invariant of conjugacy.
\end{theorem}
\begin{proof}
Let $X$, $Y$ be closed intervals, $f:X\to X$, $g:Y\to Y$ be continuous and $\pi:X\to Y$ a homeomorpishm such that $\pi \circ f = g\circ \pi$. We assume additionally that $\pi$ preserves orientation (it is increasing).

Let $M_f=\sup_{x\in X}f(x)$ and $M_g=\sup_{y\in Y}g(y)$ and let $x_M \in X$, $y_M \in Y$ be such that $M_f=f(x_M)$ and $M_g=g(y_M)$. First, we will show that $M_g=g(\pi(x_M))$. Assume it is not the case, that is
\begin{equation*}
    g(\pi(x_M))< g(y_M).
\end{equation*}
In what follows we will skip the symbol of function composition and simply write $\pi^{-1} g\pi$ instead of $\pi^{-1}\circ  g\circ \pi$.
Since $\pi$ is increasing, so is $\pi^{-1}$ and we have
\begin{equation*}
    f(x_M)=\big(\pi^{-1}g\pi\big) x_M< \big(\pi^{-1}g\big) y_M = \big(\pi^{-1} g \pi\big)(\pi^{-1}y_M) = f(\pi^{-1}y_M)
\end{equation*}
which is a contradiction. From the above we conclude that
\begin{equation}
    M_g = g(\pi x_M) = \pi(\pi^{-1}g\pi)x_M = (\pi f)x_M = \pi M_f.
\end{equation}

Similarly, if $m_f=\inf_{x\in X} f(x)$ then $m_g=\inf_{y\in Y} g(y) = \pi m_f$.

We will show now that if $x_*$ is a relevant extremum of $f$ then $y_*=\pi(x_*)$ is relevant extremum of $g$. Since $m_f\leq x_*\leq M_f$ and $\pi$ is increasing we have
\begin{equation*}
    m_g\leq y_*\leq M_g.
\end{equation*}
There remains to show that $y_*$ is a proper extremum of $g$. Indeed, if $U_f$ is an open interval containing $x_*$ such that $f(x)<f(x_*)$ for $x\in U_f\setminus\{x_*\}$ then $U_g=\pi U_f$ is an open interval in $Y$ containing $y_*$ and for $y\in U_g\setminus\{y_*\}$ we have $y=\pi x$ for some $x\in U_f\setminus\{x_*\}$ and
\begin{equation*}
    g(y) = \pi(\pi^{-1}g\pi)x = \pi f(x) < \pi f(x_*) =  (\pi f \pi^{-1})\pi x_* = g(y_*).
\end{equation*}
Similarly we argue that if $x_*$ is a proper minimum for $f$ then so is $\pi x_*$ for $g$.

Summarizing, we have shown that $\mathrm{relEx}(f)\leq \mathrm{relEx}(g)$. Since the conjugacy relation is symmetric, we can repeat the arguments for $\pi^{-1}$ and conclude that $\mathrm{relEx}(g)\leq \mathrm{relEx}(f)$.

The proof for the case of decreasing $\pi$ goes similarly, although the role of minima and maxima must be exchanged.
\end{proof}

\begin{figure}[htbp]
  \centerline{\includegraphics[width=\textwidth]{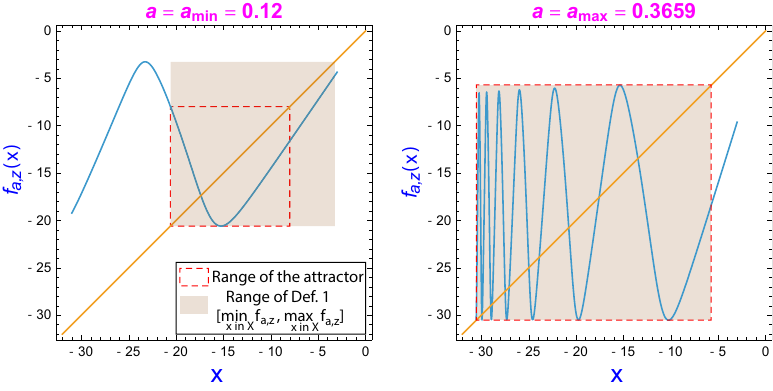}}
  \caption{Plot of the function $f_{a,z}:[\xmin,\xmax]\to[\xmin,\xmax]$ for $z=z_{\mathrm{mid}}:=\frac{1}{2}(\zmin+\zmax)$ and
	$a=\amin$ (left) and $a=\amax$ (right). It is a numerical evidence that $\mathrm{relEx}(f_{\amin,z_{\mathrm{mid}}})=1$ and $\mathrm{relEx}(f_{\amax,z_{\mathrm{mid}}})=12$. We show in a brown rectangle the range of  Def.~\ref{def:relEx}($\left[\min_{x\in X} f_{a,z}, \, \max_{x\in X} f_{a,z}\right]$). To compare with the range of the attractor (see bottom plots of Fig.~\ref{fig:AttractorAndSection}), we plot them as a red rectangle.
	 \label{fig:RelevantExtrema}}
\end{figure}


\subsection{Saddle-node bifurcations.}
In the sine model (\ref{eq:SinModel}) we observed, that some of relevant extrema (in particular, the maxima) are created via saddle-node bifurcations. Such bifurcation creates a stable periodic orbit, which coexists with chaotic and repelling invariant set. As can be seen in Fig.~\ref{fig:c15}, this is also observed in the R\"ossler model (\ref{eq:rossler}). In this section we will focus on this scenario.

Consider the following function.
\begin{equation}\label{eq:BifEquation}
	g(x,z,a) = \left(P_a(x,z)-(x,z),\mathrm{det}(DP_a(x,z)-\mathrm{Id})\right).
\end{equation}
Clearly $g(x,z,a)=0$ if $(x,z)$ is a fixed point of $P_a$ and $\lambda=1$ is an eigenvalue of $DP_a(x,z)$. Using standard Newton method applied to $g$ we have found  approximate zeroes of $g$ in the parameter range $[\amin,\amax]$ -- see Table~\ref{tab:BifPoints}.

\begin{table}[htbp]
\begin{center}
\begin{tabular}{c|c|c|c}
		$i$ & $x_i$ & $z_i$ & $a_i$\\ \hline
		1 & -24.98615641824929 & 0.005003695953767296 & 0.2445890212249042\\ \hline
		2 & -27.90101006311546 & 0.004663547021688063 & 0.3119866509093180\\ \hline
		3 & -29.22211573599117 & 0.004524155280447693 & 0.3405236989996505\\ \hline
		4 & -29.89377365336397 & 0.004456435009829962 & 0.3546641965836549\\ \hline
		5 & -30.25817997684925 & 0.004420535075303823 & 0.3623180183723328
\end{tabular}
\vskip\baselineskip
\caption{Approximate bifurcation points $(x_i,z_i,a_i)$.\label{tab:BifPoints}}
\end{center}
\end{table}

\begin{theorem}\label{thm:bifurcation}
	The Poincar\'e map $P_a$ undergoes saddle-node bifurcation at the points $(x_i^*,z_i^*,a_i^*)$, $i=1,\dots,5$
	with
	\begin{equation}\label{eq:BifPointBounds}
	|x_i^*-x_i|\leq 3\cdot 10^{-11},\qquad |z_i^*-z_i|\leq 10^{-14},\qquad |a_i^*-a_i|\leq 6\cdot 10^{-14},
	\end{equation}
	where $(x_i,z_i,a_i)$ are listed in Table~\ref{tab:BifPoints}. 	Moreover, $\mathrm{Sp}\left(DP_{a_i^*}(x_i^*,z_i^*)\right)=\{1,\lambda_i\}$ and $|\lambda_i|<2\cdot 10^{-4}$.
	
	For $i=1,\dots,5$ there are two different and continuous branches of fixed points $L_i(a)$, $R_i(a)$ parameterized by $a\in[a_i^*,\amax]$ such that $L_i(a_i^*)=R_i(a_i^*)=(x_i^*,z_i^*)$ and $L_i(a)\neq R_i(a)$ for $a>a_i^*$.
	
\end{theorem}	
In Fig.~\ref{fig:BifDiagram} we present bifurcation diagram of the fixed points of $P_a$ resulting from Theorem~\ref{thm:bifurcation}. The turning points are the points of the saddle-node bifurcation leading to a stable and an unstable branch of fixed points of $P_a$.
\begin{figure}[htbp]
	\centerline{\includegraphics[width=.75\textwidth]{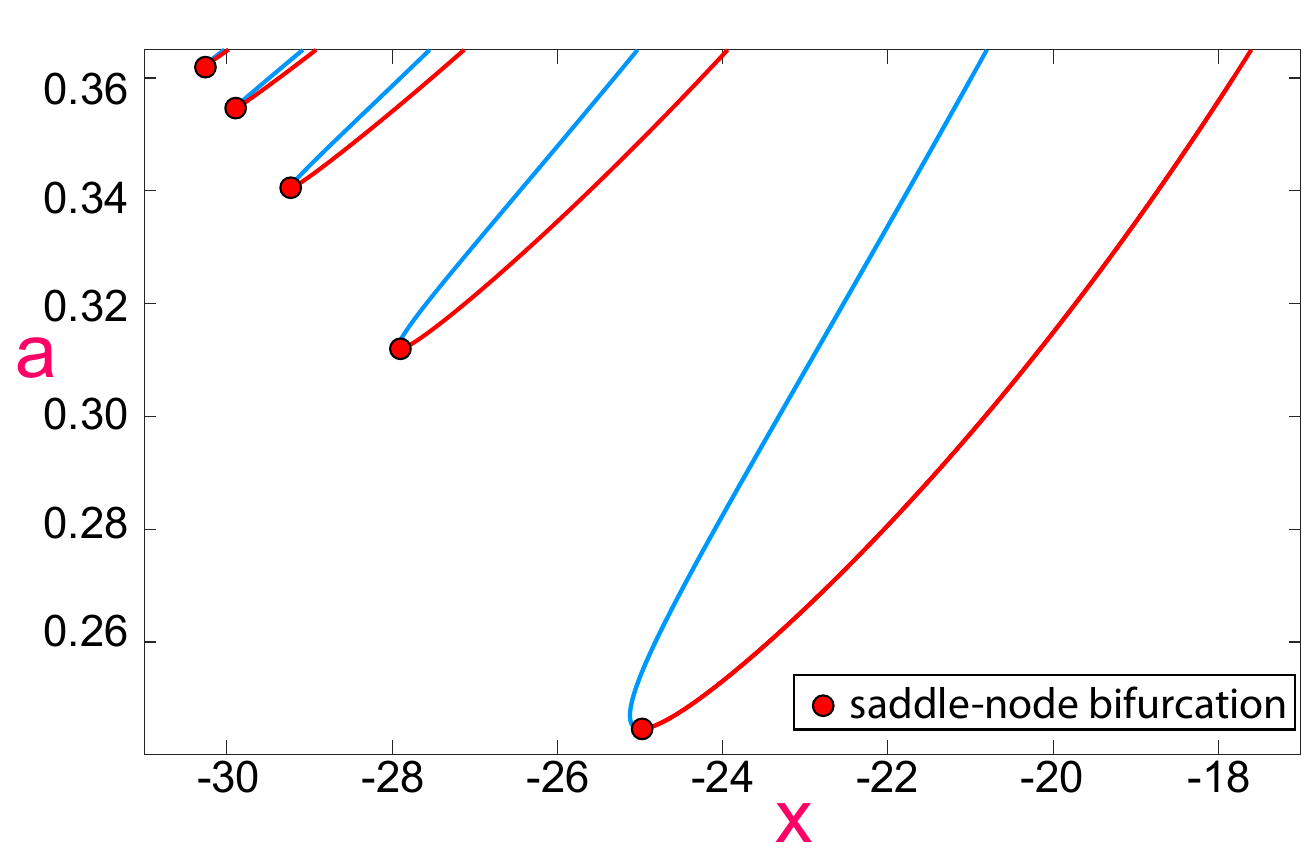}}
	\caption{Saddle-node bifurcations and branches of fixed points of $P_a$ resulting from Theorem~\ref{thm:bifurcation}.\label{fig:BifDiagram}}
\end{figure}
Since the absolute value of the non-bifurcation eigenvalue of $DP_{a_i^*}(x_i^*,z_i^*)$ at each bifurcation point is less than $1$, for parameter values slightly above $a_i^*$, $i=1,\dots,5$, an attracting periodic orbit is created and it coexists with unstable chaotic invariant set -- see Fig.~\ref{fig:c15}.


\subsection{Symbolic dynamics and topological entropy.}
\label{subsec:entropy}
In the scenario presented for the sine model (\ref{eq:SinModel}) we have seen that increasing number of relevant extrema leads to growth of the topological entropy of the system. Here we address the same question for the Poincar\'e map $P_a$.

The following definitions are standard (see for example
\cite{GH,Katok_Hasselblatt_1995}). Let us fix $k>0$ and let
$\{M_{ij}\}_{i,j=1,\dots,k}$ be $k \times k$ matrix, such that
$M_{ij} \in \{0,1\}$. We define $\Sigma_M$ by
\begin{eqnarray}\label{eq:transitionMatrix}
	\Sigma_M &=& \{ c \in \{1,2,\dots,k\}^\mathbb{Z} \ | \ M_{c_i c_{i+1}}=1 \
	\forall i \in \mathbb{Z}
	\}.
\end{eqnarray}
We define a shift map $\sigma$ on $\Sigma_M$ by
\begin{equation*}
	\sigma(c)_i= c_{i+1}, \quad \forall i \in \mathbb{Z}.
\end{equation*}
The pair $(\Sigma_M,\sigma)$ is called \emph{subshift of finite type with transition matrix $M$}.

The shift dynamics can be easily visualized on finite graphs. The constant $k$ is often called \emph{the number of symbols} but can be seen as the number of vertices in a directed graph. The transition matrix $M$ defines edges in this graph ($M_{ij}=1$ iff there is an edge from vertex $i$ to vertex $j$). A biinfinite sequence $(c_i)_{i\in\mathbb Z}\in \Sigma_M$ defines a biinfinite path in this graph. Clearly, the complexity of the shift dynamics (number of different trajectories or different possible paths on the graph) increases when we add new edges to the graph (new nonzero coefficient in $M$).

The following theorem is a classical result about entropy of topological Markov chains.

\begin{theorem}[{\cite[Prop.~3.2.5]{Katok_Hasselblatt_1995}}]\label{thm:shiftEntropy}
	The topological entropy of the shift map $(\Sigma_M,\sigma)$ is equal to
	$$
	h_{top}(\sigma) = h_{top}(M)=\max_{\lambda\in\mathrm{Sp}(M)}\log(|\lambda|).
	$$
\end{theorem}	
It is well known that the topological entropy is an invariant of conjugacy of maps. In the case of semiconjugacy we obtain only one-side inequality. Thus, showing semiconjugacy of a map $f$ to the shift dynamics $\sigma$ is a way to obtain a lower bound of the topological entropy $h_{top}(f)\geq h_{top}(\sigma)$.

\begin{theorem}\label{thm:entropy}
For all parameter values $a\in[\amin,\amax]$ there is an invariant subset $\mathcal H_a\subset \mathcal T$ for $P_a$, such that $P_a|_{\mathcal H_a}$ is semiconjugated to a subshift of finite type. The number of symbols in symbolic varies from $2$ for $a=\amin$ to $13$ for $a=\amax$. A~lower bound for the topological entropy of $P_a$ in different parameter ranges is listed in Table~\ref{tab:entropy}.
\end{theorem}

\begin{table}[htbp]
	\centering \tiny
	\begin{tabular}{c|c|c|c||c|c|c|c||c|c|c|c}
	$i$ &	$a_i$ & $k_i$ & $h_{top}(\sigma_i)$ & $i$ & $a_i$ & $k_i$ & $h_{top}(\sigma_i)$ & $i$ & $a_i$ & $k_i$ & $h_{top}(\sigma_i)$\\ \hline\hline
$1$ & $0.12$ & $2$ & $0.69424$ & $30$ & $0.342225$ & $8$ & $2.79881$ & $59$ & $0.360131$ & $11$ & $3.29216$\\ \hline
$2$ & $0.205347$ & $3$ & $1$ & $31$ & $0.34285$ & $8$ & $2.83202$ & $60$ & $0.360669$ & $11$ & $3.30727$\\ \hline
$3$ & $0.224948$ & $3$ & $1.27155$ & $32$ & $0.34345$ & $8$ & $2.86293$ & $61$ & $0.360981$ & $11$ & $3.32192$\\ \hline
$4$ & $0.245266$ & $3$ & $1.44998$ & $33$ & $0.344494$ & $8$ & $2.89187$ & $62$ & $0.361279$ & $11$ & $3.33614$\\ \hline
$5$ & $0.248131$ & $4$ & $1.44998$ & $34$ & $0.345848$ & $8$ & $2.9191$ & $63$ & $0.361464$ & $11$ & $3.34995$\\ \hline
$6$ & $0.26019$ & $4$ & $1.68451$ & $35$ & $0.347179$ & $8$ & $2.93608$ & $64$ & $0.361652$ & $11$ & $3.36338$\\ \hline
$7$ & $0.268539$ & $4$ & $1.8325$ & $36$ & $0.349673$ & $9$ & $2.95264$ & $65$ & $0.361687$ & $11$ & $3.34995$\\ \hline
$8$ & $0.282119$ & $4$ & $1.8325$ & $37$ & $0.350887$ & $9$ & $2.97691$ & $66$ & $0.361737$ & $11$ & $3.36338$\\ \hline
$9$ & $0.293257$ & $5$ & $1.89996$ & $38$ & $0.351899$ & $9$ & $3$ & $67$ & $0.361862$ & $11$ & $3.37645$\\ \hline
$10$ & $0.299273$ & $5$ & $2$ & $39$ & $0.352489$ & $9$ & $3.02203$ & $68$ & $0.362$ & $11$ & $3.38918$\\ \hline
$11$ & $0.304501$ & $5$ & $2.08272$ & $40$ & $0.353056$ & $9$ & $3.0431$ & $69$ & $0.362107$ & $11$ & $3.40159$\\ \hline
$12$ & $0.307588$ & $5$ & $2.15363$ & $41$ & $0.353415$ & $9$ & $3.06331$ & $70$ & $0.362319$ & $11$ & $3.4137$\\ \hline
$13$ & $0.312085$ & $5$ & $2.21591$ & $42$ & $0.353793$ & $9$ & $3.08272$ & $71$ & $0.362335$ & $12$ & $3.4137$\\ \hline
$14$ & $0.312612$ & $6$ & $2.21591$ & $43$ & $0.354069$ & $9$ & $3.10141$ & $72$ & $0.36252$ & $12$ & $3.42792$\\ \hline
$15$ & $0.315698$ & $6$ & $2.29786$ & $44$ & $0.354117$ & $9$ & $3.08272$ & $73$ & $0.362596$ & $12$ & $3.44169$\\ \hline
$16$ & $0.317262$ & $6$ & $2.36634$ & $45$ & $0.354156$ & $9$ & $3.10141$ & $74$ & $0.362686$ & $12$ & $3.45506$\\ \hline
$17$ & $0.31969$ & $6$ & $2.42553$ & $46$ & $0.354669$ & $9$ & $3.11942$ & $75$ & $0.36276$ & $12$ & $3.46804$\\ \hline
$18$ & $0.322604$ & $6$ & $2.47787$ & $47$ & $0.354711$ & $10$ & $3.11942$ & $76$ & $0.362864$ & $12$ & $3.48066$\\ \hline
$19$ & $0.326625$ & $6$ & $2.51135$ & $48$ & $0.355138$ & $10$ & $3.14117$ & $77$ & $0.36296$ & $12$ & $3.49295$\\ \hline
$20$ & $0.330221$ & $6$ & $2.47024$ & $49$ & $0.355319$ & $10$ & $3.16189$ & $78$ & $0.363111$ & $12$ & $3.50491$\\ \hline
$21$ & $0.331218$ & $7$ & $2.49716$ & $50$ & $0.355538$ & $10$ & $3.1817$ & $79$ & $0.36327$ & $12$ & $3.51658$\\ \hline
$22$ & $0.3337$ & $7$ & $2.5431$ & $51$ & $0.355729$ & $10$ & $3.20067$ & $80$ & $0.363562$ & $12$ & $3.52796$\\ \hline
$23$ & $0.335811$ & $7$ & $2.58496$ & $52$ & $0.356014$ & $10$ & $3.21888$ & $81$ & $0.363962$ & $12$ & $3.53908$\\ \hline
$24$ & $0.337042$ & $7$ & $2.62346$ & $53$ & $0.356308$ & $10$ & $3.2364$ & $82$ & $0.364342$ & $12$ & $3.54654$\\ \hline
$25$ & $0.338258$ & $7$ & $2.65915$ & $54$ & $0.356837$ & $10$ & $3.25328$ & $83$ & $0.365012$ & $13$ & $3.55392$\\ \hline
$26$ & $0.339076$ & $7$ & $2.69244$ & $55$ & $0.357549$ & $10$ & $3.26958$ & $84$ & $0.365385$ & $13$ & $3.56449$\\ \hline
$27$ & $0.340544$ & $7$ & $2.72367$ & $56$ & $0.358474$ & $10$ & $3.28097$ & $85$ & $0.365689$ & $13$ & $3.57483$\\ \hline
$28$ & $0.340679$ & $8$ & $2.72367$ & $57$ & $0.35939$ & $10$ & $3.26655$ & $86$ & $0.365865$ & $13$ & $3.58496$\\ \hline
$29$ & $0.341748$ & $8$ & $2.76289$ & $58$ & $0.359478$ & $11$ & $3.27655$ & \\ \hline
\end{tabular}
\vskip\baselineskip
\caption{\label{tab:entropy}
	A lower bound of the topological entropy of $P_a$ in different subintervals of the parameter range $[\amin,\amax]$. For the parameter values $a\in[a_i,a_{i+1}]$ with $a_{87}=\amax$, $P_a|_{\mathcal H_a}$ is semiconjugated to a shift dynamics $\sigma_i$ on $k_i$ symbols with  topological entropy $h_{top}(\sigma_i)$. See also Fig.~\ref{fig:c15} (top panel).}
\end{table}

\begin{remark}
	The data in Table~\ref{tab:entropy} returned by the validation algorithm (described in Section~\ref{sec:CAP}) is a lower bound for the complexity of dynamics of $P_a$ restricted to some invariant set $\mathcal H_a$ (not necessarily the maximal invariant set). Based on the approximate minima of the first iterate of $f_{a,z}$ given by (\ref{eq:faz}) the algorithm constructs and validates semiconjugacy of $P_a$ to a subshift of finite type. In order to apply computer-assisted reasoning we need a margin for accumulated errors coming from overestimation appearing in validated integration of ODEs and, most important, the fact that the algorithm always proceeds an interval of parameters rather than a single parameter. Therefore, the changes in the topological entropy observed in non-validated numerical simulation appear always for slightly smaller values of parameter than those presented in Table~\ref{tab:entropy}. Considering higher order iterates of $P_a$ would, perhaps, return a more accurate lower bound on topological entropy while requiring very large CPU time.
\end{remark}	

List of all transition matrices
 $M_i$, $i=1,\ldots,86$ is available in the supplementary material to this article \cite{github}. For $a=\amin$ the transition matrix is equal to
\begin{equation*}
	M_{1} = \begin{bmatrix} 0 & 1 \\ 1 & 1\end{bmatrix}
\end{equation*}
with $h_{top}(M_{1}) = \log_2\frac{1+\sqrt{5}}{2}\approx0.6942419136306173$. When $a$ grows, the number of symbols increases and the leading transition matrices in the sequence are
\begin{equation*}\scriptsize
	M_{2} = \begin{bmatrix} 0 & 1 & 1\\ 1 & 1 &1 \\ 1 & 0 & 0\end{bmatrix}, \quad 	
	M_{3} = \begin{bmatrix} 0 & 1 & 1\\ 1 & 1 &1 \\ 1 & 1 & 0\end{bmatrix}, \quad
	M_{4} = \begin{bmatrix} 0 & 1 & 1\\ 1 & 1 &1 \\ 1 & 1 & 1\end{bmatrix}, \quad	
	M_{5} = \begin{bmatrix} 0 & 1 & 1 & 1\\ 1 & 1 & 1 & 1 \\ 1 & 1 & 1 & 1 \\ 0 & 0 & 0 & 1\end{bmatrix}, \ldots
\end{equation*}
For $a\in[a_{86},\amax]$ the algorithm returned semiconjugacy of $P_a|_{\mathcal H_a}$ to a subshift of finite type with the transition matrix equal to 
\begin{equation}\label{eq:M86}\scriptsize
	M_{86} =
\left[
\begin{array}{ccccccccccccc}
	0 & 0 & 0 & 0 & 1 & 1 & 1 & 1 & 1 & 1 & 1 & 1 & 1 \\
	1 & 1 & 1 & 1 & 1 & 1 & 1 & 1 & 1 & 1 & 1 & 1 & 1 \\
	1 & 1 & 1 & 1 & 1 & 1 & 1 & 1 & 1 & 1 & 1 & 1 & 1 \\
	1 & 1 & 1 & 1 & 1 & 1 & 1 & 1 & 1 & 1 & 1 & 1 & 1 \\
	1 & 1 & 1 & 1 & 1 & 1 & 1 & 1 & 1 & 1 & 1 & 1 & 1 \\
	1 & 1 & 1 & 1 & 1 & 1 & 1 & 1 & 1 & 1 & 1 & 1 & 1 \\
	1 & 1 & 1 & 1 & 1 & 1 & 1 & 1 & 1 & 1 & 1 & 1 & 1 \\
	1 & 1 & 1 & 1 & 1 & 1 & 1 & 1 & 1 & 1 & 1 & 1 & 1 \\
	1 & 1 & 1 & 1 & 1 & 1 & 1 & 1 & 1 & 1 & 1 & 1 & 1 \\
	1 & 1 & 1 & 1 & 1 & 1 & 1 & 1 & 1 & 1 & 1 & 1 & 1 \\
	1 & 1 & 1 & 1 & 1 & 1 & 1 & 1 & 1 & 1 & 1 & 1 & 1 \\
	1 & 1 & 1 & 1 & 1 & 1 & 1 & 1 & 1 & 1 & 1 & 1 & 1 \\
	1 & 1 & 1 & 1 & 0 & 0 & 0 & 0 & 0 & 0 & 0 & 0 & 0 \\
\end{array}
\right]
\end{equation}
and with $h_{top}(M_{86}) = \log_2 12\approx3.5849625007211562$.


\section{Computer-assisted proofs of main results.}
\label{sec:CAP}

Due to the impossibility to obtain analytical proofs, the proofs of Theorems~\ref{thm:attractor},~\ref{thm:template},~\ref{thm:bifurcation} and ~\ref{thm:entropy} are computer assisted --  that is we used a~computer to obtain guaranteed bounds on Poincar\'e map $P_a$ and its derivatives with respect to arguments and the parameter. From these bounds we conclude the assertions of Theorems~\ref{thm:attractor},~\ref{thm:template},~\ref{thm:bifurcation} and ~\ref{thm:entropy}.

We used the CAPD library \cite{capd}, which is a general purpose C++ tool for rigorous numerical analysis of dynamical systems. The library implements algorithms for integration of (higher order) variational equations for ODEs \cite{C1HO,CnLohner,C1Lohner} as well as computation of bounds on Poincar\'e maps and their derivatives \cite{poincare}.

We would like to emphasize, that the computations related to Theorems~\ref{thm:attractor},~\ref{thm:template},~\ref{thm:bifurcation} and ~\ref{thm:entropy} are quite demanding -- total time of computation was about 6 hours on a computer running $360$ parallel threads. Most of the time, however, was spent to obtain bounds for parameter range $a\geq 0.34$. Below this value, the program could be easily run on a laptop computer with 16 CPUs.


\subsection{Proof of Theorem~\ref{thm:attractor}.}
Recall, the parameter range $\mathcal A$ is defined in (\ref{eq:arange}) and the trapping region $\mathcal T$ is defined in (\ref{eq:defTrappingRegion}). We have to show that for $a\in[\amin,\amax]$ the Poincar\'e map (\ref{eq:PoincareMap}) exists on $\mathcal T$ and $P_a(\mathcal T)\subset\mathrm{int}\mathcal T$.

Validation of the inclusion $P_a(\mathcal T)\subset\mathrm{int}\mathcal T$ is split into two steps.

\textbf{Step 1.}
	First we validate that for all $a\in \mathcal A$ and for all $u\in\mathcal T$ the Poincar\'e map $P_a(u)$ is defined. Thus no restriction on obtained bounds on $P_a(u)$ is given. For this purpose we cover the set $\mathcal A\times \mathcal T$ by some initial grid of boxes $A_i\times X_i\times Z_i$, $i=1,\dots,N$. Then for each $i$ we call a general routine from the CAPD library \cite{capd} that computes bound on $P_{A_i}(X_i,Z_i)$. By the construction of the algorithm from the CAPD library, if the procedure returns (any) bound, then the Poincar\'e map exists, and by implicit function theorem it is smooth on its domain. Otherwise, the algorithm throws an exception. In this case, we bisect the set $A_i\times  X_i\times Z_i$ in $(a,x)$ coordinates and repeat computation.
	
	Such subdivision is repeated until existence of Poincar\'e map is validated on each set in the subdivision or the maximal depth of subdivision is exceeded. In the last case we return \texttt{Failure} and stop computation.
	
	Running this algorithm we found a (non-uniform) cover of $\mathcal A\times \mathcal T $, which consists of $36\,316\,641$ boxes, on which the existence of Poincar\'e map has been validated.
	
\textbf{Step 2.}
	In the second step  we check that $P_a(\partial\mathcal T)\subset \mathrm{int}\mathcal T$ for $a\in \mathcal A$. After adaptive subdivision, as in the first step, we found a (non-uniform) cover $A_i\times X_i\times Z_i$, $i=1,\dots,9\,404\,150$ of $\mathcal A\times \partial\mathcal T$, such that $P_{A_i}(X_i\times Z_i)\subset \mathrm{int}\mathcal T$ for $i=1,\dots,9\,299\,383$.
	
	Given that $P_a$ is a diffeomorphism onto image, by the Jordan Theorem we conclude that $P_a(\mathcal T)\subset \mathrm{int}\mathcal T$, for $a\in\mathcal A$.
\qed

In Fig.~\ref{fig:enclosure} we show obtained bound on $P_{a}(\partial\mathcal T)$ for $a=\amin$ and $a=0.325$.
\begin{figure}
  \centerline{\includegraphics[width=.48\textwidth]{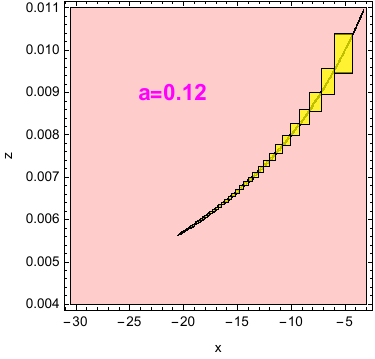}\includegraphics[width=.48\textwidth]{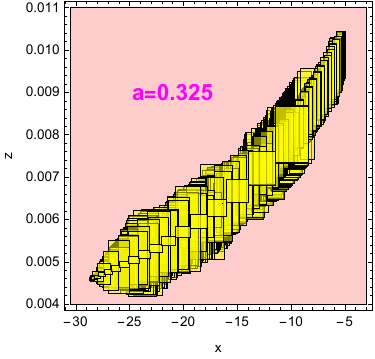}}
	\caption{Trapping region $\mathcal T$ is shown in pink. Computed bound on $P_{a}(\partial\mathcal T)$ for $a=\amin$ (left) and $a=0.325$ (right) is shown in yellow.
	\label{fig:enclosure}}
\end{figure}


\subsection{Proof of Theorem~\ref{thm:template}.}
Put $X=[\xmin,\xmax]$ and $Z=[\zmin,\zmax]$. In Theorem~\ref{thm:template} we have to count the number of relevant extrema of the map $f_{a,z}:X\to X$, which is defined in (\ref{eq:faz}). For each parameter value (see (\ref{eq:templateParameters}))
$$a\in \{\amin, 0.2 ,0.26, 0.3, 0.32, 0.333, 0.345, 0.35, 0.356, 0.36, 0.362, \amax\}$$
we run an algorithm, which consists of the following steps.

\begin{itemize}
	\item Using standard bisection search we localize approximate extrema of $f_{a,z}$ in $X$ for $z\in\{\zmin,\zmax,\frac{1}{2}(\zmin+\zmax)\}$. Denote this (finite) set by $E$.
	\item We initially cover $X$ by subintervals $X_i$.
	\item If $X_i\cap E=\emptyset$, we compute bound on $f_{a,Z}(X_i)$ and $f'_{a,Z}(X_i)$. If computed bound on $f'_{a,Z}(X_i)$ contains zero, the interval $X_i$ is bisected and the procedure repeats until bound on derivative does not contain zero or the subdivision depth is exceeded.
	\item Similarly, if $X_i\cap E\neq \emptyset$ we request computation of bounds on $f_{a,Z}(X_i)$, $f'_{a,Z}(X_i)$ and $f''_{a,Z}(X_i)$.
	 If computed bound on $f''_{a,Z}(X_i)$ contains zero, the interval $X_i$ is bisected and the procedure repeats until bound on the second derivative does not contain zero or the subdivision depth is exceeded.
\end{itemize}
By the construction of the algorithm, if it stops and all tasks return nonzero bounds on first or second derivative of $f_{a,Z}$, respectively, the domain $X$ is covered by intervals $Y_i$ on which either $f_{a,Z}$ is monotone or convex/concave.

A sample output of the algorithm for $a=\amax$ is given in Table~\ref{tab:template}. From the bounds on $f'_{a,z}$ and $f''_{a,z}$ we see that the function has exactly $12$ local extrema in $X$. From the bound on $f_{a,z}$ (second column) and from Theorem~\ref{thm:attractor} we obtain
$$-30.53\leq \inf_{x\in X} f_{a,Z}(x)\leq -30.52\quad\text{and}\quad -5.75\leq \sup_{x\in X} f_{a,Z}(x)\leq -5.73.$$
From these bounds we conclude that all extrema found are relevant extrema and thus $\mathrm{relEx}(f_{\amax,z})= 12$ for $z\in[\zmin,\zmax]$.

Output of the program for remaining parameter values (\ref{eq:templateParameters}) is given in the supplementary material \cite{github}. From obtained bounds we can conclude the number of relevant extrema of $f_{a,z}$ as given in (\ref{eq:templateParameters}).\qed

\begin{remark}
	Note, that computation of $f''_{a,z}$ requires costly integration of the second order variational equation for (\ref{eq:rossler}).
\end{remark}

\begin{table}[htbp]
	\centering \scriptsize
	\begin{tabular}{c|c|c|c}
		$x$ & $f_{a,z}(x)$ & $f'_{a,z}(x)$ & $f''_{a,z}(x)$ \\ \hline\hline
$[-30.530000000000000, -30.234918150004624]$ & $[-30.19, -6.57]$ & $[0.001, 184]$ & $-$\\ \hline
$[-30.234918150004624, -30.234487860323142]$ & $[-6.58, -6.57]$ & $[-0.72, 0.58]$ & $[-1326, -325]$\\ \hline
$[-30.234487860323142, -29.92685079785111]$ & $[-30.44, -6.57]$ & $[-230, -0.001]$ & $-$\\ \hline
$[-29.92685079785111, -29.925276370697883]$ & $[-30.44, -30.43]$ & $[-0.89, 2.63]$ & $[101, 1193]$\\ \hline
$[-29.925276370697883, -29.433781614317695]$ & $[-30.44, -6.47]$ & $[0.0007, 122]$ & $-$\\ \hline
$[-29.433781614317695, -29.432550550415712]$ & $[-6.48, -6.47]$ & $[-0.77, 0.59]$ & $[-578, -68]$\\ \hline
$[-29.432550550415712, -28.936271584068677]$ & $[-30.45, -6.47]$ & $[-138, -0.009]$ & $-$\\ \hline
$[-28.936271584068677, -28.933364707472052]$ & $[-30.45, -30.44]$ & $[-0.87, 1.87]$ & $[73, 418]$\\ \hline
$[-28.933364707472052, -28.134274742819773]$ & $[-30.45, -6.36]$ & $[0.006, 75.5]$ & $-$\\ \hline
$[-28.134274742819773, -28.129878233661032]$ & $[-6.37, -6.36]$ & $[-1.2, 0.66]$ & $[-227, -18]$\\ \hline
$[-28.129878233661032, -27.315443495151083]$ & $[-30.47, -6.36]$ & $[-86.7, -0.04]$ & $-$\\ \hline
$[-27.315443495151083, -27.309055771711787]$ & $[-30.47, -30.45]$ & $[-1.66, 1.72]$ & $[23, 156]$\\ \hline
$[-27.309055771711787, -25.978724324556815]$ & $[-30.49, -6.23]$ & $[0.004, 31.4]$ & $-$\\ \hline
$[-25.978724324556815, -25.969275484986035]$ & $[-6.24, -6.23]$ & $[-0.83, 0.73]$ & $[-79, -9.6]$\\ \hline
$[-25.969275484986035, -24.583975022884346]$ & $[-30.49, -6.23]$ & $[-49.91, -0.01]$ & $-$\\ \hline
$[-24.583975022884346, -24.573109921220805]$ & $[-30.5, -30.45]$ & $[-1.1, 1.2]$ & $[7.2, 56]$\\ \hline
$[-24.573109921220805, -22.263096727912796]$ & $[-30.48, -6.05]$ & $[0.01, 17.7]$ & $-$\\ \hline
$[-22.263096727912796, -22.244978977377045]$ & $[-6.06, -6.05]$ & $[-0.68, 0.71]$ & $[-29.2, -0.2]$\\ \hline
$[-22.244978977377045, -19.75178920778561]$ & $[-30.5, -6.05]$ & $[-21.8, -0.02]$ & $-$\\ \hline
$[-19.75178920778561, -19.732234778220185]$ & $[-30.5, -30.49]$ & $[-0.36, 0.36]$ & $[0.15, 19]$\\ \hline
$[-19.732234778220185, -15.379434543948365]$ & $[-30.5, -5.74]$ & $[0.0007, 8.1]$ & $-$\\ \hline
$[-15.379434543948365, -15.345294934267804]$ & $[-5.75, -5.73]$ & $[-0.51, 0.51]$ & $[-7.9, -0.2]$\\ \hline
$[-15.345294934267804, -10.256384670433105]$ & $[-30.53, -5.74]$ & $[-7.8, -0.004]$ & $-$\\ \hline
$[-10.256384670433105, -10.216471648755972]$ & $[-30.54, -30.52]$ & $[-0.2, 0.2]$ & $[1.6, 2.6]$\\ \hline
$[-10.216471648755972, -3]$ & $[-30.54, -9.63]$ & $[0.01, 3.3]$ & $-$\\ \hline
\end{tabular}
\vskip\baselineskip
\caption{Bounds on $f_{a,z}$ and its derivatives for $a=\amax$ and $z\in[\zmin,\zmax]$.\label{tab:template}}
\end{table}

\subsection{Proof of Theorem~\ref{thm:bifurcation}.}
Before we describe the algorithm for validation of bifurcation and continuation of fixed points of $P_a$ we recall some standard tools for validation of the existence of branches of zeroes of smooth functions.

For a smooth function $F:D\subset \mathbb R^m\times \mathbb R^n\to\mathbb R^n$ and an interval vector (Cartesian product of closed intervals) $A\times X\subset D$ we set $$[D_XF(A,X)] := \mathrm{convexHull}\{D_xF(a,x) : a\in A, x\in X\}.$$

\begin{theorem}[Interval Newton Method \cite{Neumaier_1991}]\label{thm:INM}
Let $F:D\subset \mathbb R^n\to\mathbb R^n$ be a smooth function, $X\subset D$ be an interval vector and $x_0\in \mathrm{int}X$.
	If $[D_XF(X)]$ is nonsingular and
	\begin{equation}\label{eq:IMN}
	N(F,X,x_0) := x_0 - [D_XF(X)]^{-1}\cdot F(x_0)\subset \mathrm{int}X
	\end{equation}
	then $F$ has unique zero in $x^*\in X$. Moreover, $x^*\in N(F,X,x_0)$.
\end{theorem}

The following is a straightforward extension to the case of parameter dependent functions.
\begin{theorem}[Interval Newton Method for parameterized functions {\cite[Appendix A]{WalawskaWilczak2019}}]\label{thm:ParamINM}
	Let $F:D\subset \mathbb R^m\times \mathbb R^n\to\mathbb R^n$ be a smooth function, $A\times X\subset D$ be an interval vector and $x_0\in \mathrm{int}X$.
	If $[D_XF(A,X)]$ is nonsingular and
	\begin{equation}\label{eq:ParamIMN}
	N = N(F,A,X,x_0) := x_0 - [D_XF(A,X)]^{-1}\cdot F(A,x_0)\subset \mathrm{int}X
	\end{equation}
	then there is a smooth function $g:A\to N\subset X$ such that $F(g(x),x)\equiv 0$. Moreover, if $F(a,x)=0$ for some $(a,x)\in A\times X$ then $a=g(x)$.
\end{theorem}

Let us fix parameter $a=a_i$ as in Table~\ref{tab:BifPoints}. In the description of the algorithm we provide some data (numbers) obtained in validation of the first bifurcation point $(x_1,z_1,a_1)$ -- the remaining can be found in the supplementary material \cite{github}.

The algorithm which validates the existence of a saddle-node bifurcation and the existence of two branches of fixed points as in Theorem~\ref{thm:bifurcation} consists of the following steps.

\textbf{Step 1:}
	Validation of bifurcation point. Define
	\begin{eqnarray*}
	u_i &=& (x_i,z_i,a_i),\\ 	
	W_i &=& u_i +  [-1,1]\cdot(3\cdot 10^{-11},10^{-14},6\cdot 10^{-14}).
	\end{eqnarray*}
	Let $g$ be defined as in (\ref{eq:BifEquation}). Using algorithms from the CAPD library we compute the Interval Newton Operator (\ref{eq:IMN}) and we checked $N(g,W_i,u_i)\subset \mathrm{int}W_i$. Thus, by means of Theorem~\ref{thm:INM}, there is an unique $(x_i^*,z_i^*,a_i^*)\in W_i$ such that
	$$
	P_{a^*_i}(x_i^*,z_i^*)= (x_i^*,z_i^*)\quad \text{and}\quad 1\in \mathrm{Sp}(DP_{a^*_i}(x_i^*,z_i^*)).
	$$
	Applying Gershgorin estimate to obtained bound on the derivative $DP_{a^*_i}(x_i^*,z_i^*)$ we have checked that the second eigenvalue $\lambda_i$ of $DP_{a^*_i}(x_i^*,z_i^*)$ satisfies $|\lambda_i|\leq 2\cdot 10^{-4}$.

\textbf{Step 2:}
	Validation of a short curve of fixed points of $P_a$ near $(x_i^*,z_i^*,a_i^*)$.
	
	Define a map $F(x,z,a) = P_a(x,z)-(x,z)$. We impose that the set of zeroes of this function can be parameterized locally near $(x_i^*,z_i^*,a_i^*)$ as a smooth curve of the form
	$$u(x)=(x,z(x),a(x)).$$
	In order to apply the Interval Newton Method (Theorem~\ref{thm:ParamINM}) we define an explicit set
	$$D_i=X_i\times Z_i\times A_i = (x_i,z_i,a_i) + [-1,1]\cdot(\Delta_i^x,\Delta_i^z,\Delta_i^a).$$
	Using algorithms from the CAPD library we compute the Interval Newton Operator (\ref{eq:ParamIMN}) for the map $g$ and we check that
	$$N(g,X_i,D_i,(x_i,z_i,a_i)) \subset \mathrm{int}\left(Z_i\times A_i\right).$$
	From Theorem~\ref{thm:INM} we conclude, that all zeroes of $g$ in $D_i$ form a graph of a smooth function $u(x)=(x,z(x),a(x))$, $x\in x_i +[-1,1]\Delta_i^x$.

	Actual (hand adjusted) diameters for $i=1$ are $$(\Delta_1^x,\Delta_1^z,\Delta_1^a) = (5\cdot 10^{-4}, 3\cdot 10^{-7} ,2\cdot 10^{-5})$$
	and computed bound on the Interval Newton Operator (\ref{eq:ParamIMN}) is
	\begin{equation*}
		N(g,X_1,D_1,(x_1,z_1,a_1))\subset (z_1,a_1)- [-1,1](2.1\cdot 10^{-7},1.15\cdot 10^{-5}).
	\end{equation*}	
		
\textbf{Step 3:} Validation of saddle-node bifurcation point $(x^*_i,z^*_i,a^*_i)$.
	
	From \textbf{Step 1} and \textbf{Step 2} and the uniqueness property of the Interval Newton Method we know that $a^*_i = a(x_i^*)$ and $z^*_i = z(x_i^*)$. We would like to show, that the function $a(x)$ has a unique minimum in $x^*_i$ and it is a convex function in $X_i$. For this purpose we again apply Theorem~\ref{thm:INM} to obtain tighter bounds on $a'(x_i-\Delta_i^x)$ and $a'(x_i+\Delta_i^x)$. Then we check if these derivatives are of opposite signs. For $i=1$ we obtain bounds
	\begin{equation*}
	a'(x_1-\Delta_1^x) \in [-3.3,-3.2]\cdot 10^{-7},\quad a'(x_1+\Delta_1^x) \in [3.2,3.3]\cdot 10^{-7}.
	\end{equation*}
	Finally, we check that $a''(X_i)>0$ and thus $a$ is a convex function. For $i=1$ we obtain a bound
	$$
	a''(X_1)\subset [0.05,0.08].
	$$
	Derivatives $a'$ and $a''$ are computed by differentiating the identity
	\begin{equation}\label{eq:BifIdentity}
	g(x,z(x),a(x))\equiv 0.
	\end{equation}
	From all these bounds we conclude that the function $a(x)$ has unique minimum in $X_i$, say $\widehat a_i=a(\widehat x_i)$ with corresponding $\widehat z_i = z(\widehat x_i)$. Our aim is to show that $\widehat x_i=x_i^*$. Differentiation of (\ref{eq:BifIdentity}) and $\det \left(DP_{a_i^*}(x_i^*,z_i^*)-\mathrm{Id}\right)=0$ gives $a'(x_i^*)=0$. Since $a''(X_i)>0$ the derivative $a'$ is monotone in $X_i$ and thus $0=a'(\widehat x_i)= a'(x_i^*)$ implies $x_i^*=\widehat x_i$.
	
	From the above considerations we conclude, that the set of fixed points of $P_a$ in $D_i$ can be parameterized as the union of graphs of two functions
	\begin{eqnarray*}
		L: [a_i^*,a_i(x_i-\Delta_i^x)]\ni a\to (x(a),z(a)) \in X\times Z,\\
		R: [a_i^*,a_i(x_i+\Delta_i^x)]\ni a\to (x(a),z(a)) \in X\times Z.
	\end{eqnarray*}

\textbf{Step 4:} Continuation of branches $L$ and $R$ until $a=\amax$.

	In this step we make an adaptive subdivision of the parameter ranges
	\begin{eqnarray*}
	A_i^L&=&[a_i(x_i-\Delta_i^x),\amax]\quad\text{and}\\
	A_i^R&=&[a_i(x_i+\Delta_i^x),\amax].
	\end{eqnarray*}	
	We start from an initial cover of $A_i^L$ and $A_i^R$ by intervals and we try to validate the existence of a branch of zeroes of $g$ in each subinterval using Interval Newton Operator -- see Theorem~\ref{thm:ParamINM}. If the verification step fails, we bisect the parameter range and repeat the computation. Such subdivision stops if either in each subinterval we could validate the existence of a branch of fixed points for $P_a$ or the maximal depth of subdivisions is exceeded.
	
	In each case $i=1,\dots,5$ the algorithm returned covers
	\begin{eqnarray*}
	A_i^L&=&[a_i(x_i-\Delta_i^x),\amax]\subset \bigcup_{k=1}^{K_i^L}[a_{i.k-1}^L,a_{i,k}^L]\quad\text{and}\\
	A_i^R&=&[a_i(x_i+\Delta_i^x),\amax]\subset \bigcup_{k=1}^{K_i^R}[a_{i,k-1}^R,a_{i,k}^R]
	\end{eqnarray*}	
	such that on each subinterval the existence of a branch of fixed points of $P_a$ has been validated. In each case $i=1,\dots,5$ the number subintervals $K_i^L$ and $K_i^R$ exceeds $10^4$.
	
\textbf{Step 5:} Verification of smoothness of branches $L$ and $R$.
	
	We have to check if the segments of $L$ and $R$ on each subinterval merge into a smooth curve. For this purpose we additionally compute (using again the Interval Newton Method and Theorem~\ref{thm:INM}) a tight bound on $L(a_{i,k}^L)$, $k=1,\ldots,K_L-1$ and we show that $L(a_L^k)$ belongs to computed bounds for segments $[a_{i,k-1}^L,a_{i,k}^L]$ and $[a_{i,k}^L,a_{i,k+1}^L]$. From the uniqueness property of the Interval Newton Method we conclude, that these segments merge into a continuous curve, which is smooth by the implicit function theorem. Similarly for the branch $R$.
	
	Finally, we have to repeat the argument to obtain smoothness at $a_{i,0}^L=a_i(x_i-\Delta_i^x)$ and $a_{i,0}^R=a_i(x_i+\Delta_i^x)$. This time we have to use a bound from the verification on the curve of fixed points $u(x)$ from \textbf{Step 2} and parameterized by $x\in X_i$.
	
	\qed

\subsection{Proof of Theorem~\ref{thm:entropy}.}
The proof of Theorem~\ref{thm:entropy} relies on automatic (algorithmic) construction and verification of semiconjugacy of $P_a$ to a subshift of finite type. For this purpose we use the method of covering relations introduced for two-dimensional maps by Zgliczy\'nski \cite{Z} and later extended to multidimensional case in \cite{ZGi}. It is also closely related to the method of correctly aligned windows by Easton \cite{E}.

Since $P_a$ is a two-dimensional map, we recall here the definition from \cite{Z} and simplify it to the settings of the Poincar\'e map $P_a$.
\begin{definition}
	Let $|N|=[a,b]\times [c,d]$ be a rectangle and put
	\begin{eqnarray*}
		N^{le}&=&\{a\}\times[c,d],\\
		N^{re}&=&\{b\}\times[c,d],\\
		N^{l}&=&(\infty,a)\times[c,d],\\
		N^{r}&=&(b,\infty)\times[c,d],\\
		N^s&=& \mathbb R\times (c,d).
	\end{eqnarray*}
	The tuple $N=(|N|,N^{le},N^{re},N^l,N^r,N^s)$ is called an $h$-set.
\end{definition}

Assume $N_1,\ldots, N_k$, $k\geq 1$ are pairwise disjoint $h$-sets. Put $D=\bigcup_{i=1}^k|N_i|$ and let $f:D\to\mathbb R^2$ be continuous. Denote by $\mathrm{Inv}(f,D)\subset D$ the maximal invariant set for $f$ in $D$. Because the sets are pairwise disjoint, for any $x\in \mathrm{Inv}(f,D)$ there is a unique sequence $(x_{i_j})_{j\in\mathbb Z}$ such that
\begin{itemize}
	\item $x_{i_0}=x$,
	\item $x_{i_j}\in |N_{i_j}|$, $j\in\mathbb Z$,
	\item $f(x_{i_j})=x_{i_{j+1}}$, $j\in\mathbb Z$.
\end{itemize}
The above defines a mapping $\pi:\mathrm{Inv}(f,D)\ni x\to (i_j)_{j\in\mathbb Z}\in \{1,\ldots,k\}^\mathbb {Z}$.

\begin{definition}\label{def:covrel}
	Let $f:D\subset \mathbb R^2\to\mathbb R^2$ be a continuous map and let $N_1$, $N_2$ be $h$-sets (can be the same). We say that the set $N_1$ $f$-covers $N_2$, denoted by $N_1\cover{f}N_2$, if $|N_1|\subset \mathrm{dom}(f)$, $f(|N_1|)\subset N_2^s$ and
	\begin{enumerate}
		\item either $f(N_1^{re})\subset N_2^r$ and $f(N_1^{le})\subset N_2^l$
		\item or $f(N_1^{re})\subset N_2^l$ and $f(N_1^{le})\subset N_2^r$.
	\end{enumerate}	
\end{definition}

The following theorem is a special case of result from \cite{Z,ZGi,E} about the method of covering relations.
\begin{theorem}\label{thm:covrel}
Let $N_1,\ldots, N_k$ be pairwise disjoint $h$-sets and let $M\subset \mathbb{R}^{k\times k}$ be a transition matrix defined in the following way
\begin{equation*}
	M_{ij}=\begin{cases}
		1 & \text{if } N_i\cover{f} N_j,\\
		0 & \text{otherwise}.
	\end{cases}
\end{equation*}
Put $\mathcal I = \mathrm{Inv}(f,\bigcup_{i=1}^k|N_i|)$. Then $\Sigma_M\subset \pi\big(\mathcal I)$ (see (\ref{eq:transitionMatrix}) for the definition of $\Sigma_M$). Moreover, if $c=(i_j)\in \Sigma_M$ is a periodic sequence of principal period $n$, then there is $x \in\pi^{-1}(c)\in \mathcal I$, such that $f^n(x)=x$ and $n$ is a principal period for $x$.
\end{theorem}

Theorem~\ref{thm:covrel} provides a tool for proving semiconjugacy of a map $f$ to a subshift of finite type and thus bounding from below topological entropy of maps by the topological entropy of the shift map, which is easily computable -- see Theorem~\ref{thm:shiftEntropy}. It suffices to construct suitable $h$-sets and validate covering relations between them. In the case of the R\"ossler system we succeed to do it in an automatic and algorithmic way. Below we describe the main steps of the algorithm, although it is not possible to present all heuristics that make the computation eventually completed.

\textbf{Proof of Theorem~\ref{thm:entropy}}. Let $[a_l,a_r]\subset [\amin,\amax]$ be an interval of parameters. First we construct $h$-sets for further validation of covering relations. This computation is nonrigorous and consists of the following steps.
\begin{enumerate}
	\item Set $a_m = \frac{1}{2}(a_l+a_r)$, $z_{\mathrm{mid}}=\frac{1}{2}(\zmin+\zmax)$.
	\item Find approximate extrema of $f_{a,z}$ defined in (\ref{eq:faz}) for $(a,z)\in\{a_l,a_m,a_r\}\times\{\zmin,z_{\mathrm{mid}},\zmax\}$. Notice, that the number of relevant extrema of $f_{a,z}$ may be different depending on $(a,z)$. Let $x_1>\dots>x_k$ be the approximate relevant extrema, that are present for all choices of $(a,z)$.
	\item Define $|N_i|=[x_{i}+\varepsilon,x_{i-1}-\varepsilon]\times[\zmin,\zmax]$ for $i=2,\ldots, k$, where $\varepsilon>0$ is a very small number, for example the machine epsilon $\varepsilon=2^{-52}$.
	\item From the construction, the sets $|N_i|$ are pairwise disjoint. Since $x_i$ are approximate subsequent extrema of $f_{a,z}$, we expect that  $N_{i}\cover{P_a}N_j$ for all $i,j=\{2,\ldots,k\}$.
	\item We extend the sequence $x_i$ by $x_0>x_1$ and $x_{k+1}<x_k$ and define $|N_1|=[x_1+\varepsilon,x_0]\times[\zmin,\zmax]$, $|N_{k+1}|=[x_{k+1} ,x_{k}-\varepsilon]\times[\zmin,\zmax]$. The point $x_0$ is chosen so that the image $P_a(|N_1|)$ spans across as much as possible of remaining sets $N_2,\ldots, N_{k}$ -- see Fig.~\ref{fig:covrel} and the location of $B_0$ and $B_{\text{odd}}$. It is an indication that for $a=\amax$ the image $P_{\amax}(|N_1|)$ spans across the sets $N_i$, $i=5,\ldots,12$. Similarly, $x_{k+1}$ is chosen so that the image of $P_{a}(|N_{k+1}|)$ spans across as many as possible of sets $N_1$,\ldots $N_{k+1}$. As an example, see the location of $B_{13}$ and $B_{\text{even}}$ in Fig.~\ref{fig:covrel}, which indicates that for $a=\amax$ the image $P_{a}(|N_{13}|)$ spans across $N_i$, $i=1,\ldots,4$.
\end{enumerate}
After $h$-sets $N_i$, $i=1,\ldots, k+1$ are constructed, we eventually have to rigorously compute transition matrix of covering relations. Observe, that from Theorem~\ref{thm:attractor} the condition $P_a(|N_i|)\subset N_j^s$ from Definition~\ref{def:covrel} is always satisfied for every choice of $i,j=1\dots,k+1$. Thus, in order to check $N_{i}\cover{P_a}N_j$ we have to compute bounds on $P_a(N_{i}^{le})$ and $P_a(N_{i}^{re})$ and check if
\begin{itemize}
	\item either $P_a(N_i^{re})\subset N_j^r$ and $f(N_i^{le})\subset N_j^l$
	\item or $f(N_i^{re})\subset N_j^l$ and $f(N_i^{le})\subset N_j^r$.
\end{itemize}

As an example, we present the data from the computation for a single parameter value $a=\amax$. In the first nonrigorous step, the algorithm returns the following sequence
\begin{equation}\label{eq:edges}
\begin{array}{lclclcl}
x_{0}&=&-6.579089092895479, &
x_{1}&=&-10.23628952536583,\\
x_{2}&=&-15.36224942593575, &
x_{3}&=&-19.74191129794121,\\
x_{4}&=&-22.25394390549659, &
x_{5}&=&-24.5784539999485,\\
x_{6}&=&-25.97440321955681, &
x_{7}&=&-27.31216688246727,\\
x_{8}&=&-28.13292349762917, &
x_{9}&=&-28.93522295279503,\\
x_{10}&=&-29.43325241560936, &
x_{11}&=&-29.92639805216789, \\
x_{12}&=&-30.23464037203789, &
x_{13}&=&-30.43604163408427.
\end{array}
\end{equation}
which is used to define $h$-sets $N_i$, $i=1,\ldots,13$ -- see Fig.~\ref{fig:covrel}. Then we compute bounds
$$X_i:=\pi_xP_{\amax}\big(x_i+[-2^{-52},2^{-52}],[\zmin,\zmax]\big),\quad i=0,\ldots, 13$$
and we obtain
\begin{equation}\label{eq:boundX}
    \begin{array}{lcl}
 X_{0} & \subset & [-21.0806926149584, -21.07966553069346], \\
X_{1} & \subset & [-30.52865001671991, -30.52864985142989], \\
X_{2} & \subset & [-5.74000870675155, -5.740008372417317], \\
X_{3} & \subset & [-30.49586235784065, -30.49586197137184], \\
X_{4} & \subset & [-6.053862270340931, -6.053861581770914], \\
X_{5} & \subset & [-30.47672694294398, -30.47672596112965], \\
X_{6} & \subset & [-6.230878793151707, -6.230877235952661], \\
X_{7} & \subset & [-30.46181095509974, -30.46180845922029], \\
X_{8} & \subset & [-6.363335405504911, -6.363331754784139], \\
X_{9} & \subset & [-30.44856969865465, -30.44856340550682], \\
X_{10} & \subset & [-6.475174472584746, -6.475165785873566], \\
X_{11} & \subset & [-30.43605932780109, -30.43604364795888], \\
X_{12} & \subset & [-6.575428749700791, -6.575407979118147], \\
X_{13} & \subset & [-22.77372275019209, -22.75909669369429], \\
\end{array}
\end{equation}
The location of line segments $\{x_i\}\times[\zmin,\zmax]$ and the corresponding bounds $B_i=X_i\times Z=P_{\amax}\big(x_i+[-2^{-52},2^{-52}],[\zmin,\zmax]\big), i=0,\ldots, 13$ is shown in Fig.~\ref{fig:covrel}.

From (\ref{eq:edges}) and (\ref{eq:boundX}) we see that
\begin{itemize}
	\item $X_i>x_0$ for $i=2,4,\ldots,12$,
	\item $X_i<x_{13}$ for $i=1,3,\ldots,11$,
	\item $X_0>x_4$ and
	\item $X_{13}<x_4$.
\end{itemize}
	From these inequalities we conclude that $N_i\cover{P_{\amax}}N_j$ if $i\neq 1$, $i\neq 13$ and $j=1,\ldots,12$. From the bounds $X_0,X_1,X_{12}$ and $X_{13}$ we also see that $N_1\cover{P_{\amax}}N_j$ for $j\geq 5$ and $N_{13}\cover{P_{\amax}}N_j$ for $j=1,\ldots,4$. According to Theorem~\ref{thm:covrel} we obtain that $P_{\amax}$ in restriction to $\mathrm{Inv}\big(P_{\amax},\bigcup_{i=1}^{13}|N_i|\big)$ is semiconjugated to a shift dynamics with transition matrix $M_{86}$ as defined in (\ref{eq:M86}).

Running the above algorithm with an adaptive subdivision of the parameter range $[\amin,\amax]$ we obtain semiconjugacy of $P_a$ on the corresponding invariant set to a subshift of finite type with topological entropy as listed in Table~\ref{tab:entropy}. \qed

\begin{figure}[htbp]
  \centerline{\includegraphics[width=.75\textwidth]{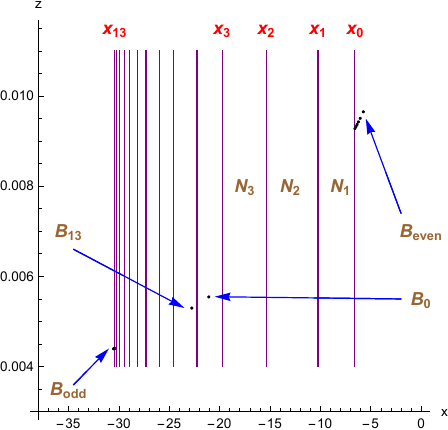}}
  \caption{Location of the edges $x_i$ (see (\ref{eq:edges})). Rigorous tiny bounds $B_i=X_i\times Z=P_{\amax}\big(x_i+[-2^{-52},2^{-52}],[\zmin,\zmax]\big), i=0,\ldots, 13$ are marked by dots -- see also (\ref{eq:boundX}).\label{fig:covrel}}
 \end{figure}

\section{Conclusions}

The concept of topological entropy is one of the most important topological invariants in dynamical systems theory, as it is one of the possible ways to measure the complexity of the dynamics.
Some landmark results by Milnor and Thurston focus on the theoretical understanding of topological entropy in maps (discrete-time systems).
Using the classical R\"ossler system as an example, we have presented an algorithmic approach to demonstrate the existence of global changes of the form of attractors in continuous-time systems and to rigorously prove lower bounds for topological entropy values. This has allowed us to show their growth when a system parameter is changed. We have proved that there exists a sequence of saddle-node bifurcations, which give rise to the semiconjugacy of a certain Poincar\'e map to symbolic dynamics in symbols from $2$ to $13$. This sequence of bifurcations leads to a growth of topological entropy. Due to the impossibility of obtaining pure analytic proofs, all theorems are obtained using computer-assisted techniques. Furthermore, we show that for an explicit range of parameters there exists a chaotic attractor.

\section*{Acknowledgments}
RB and SS have been supported by the Spanish Research projects PID2021-122961NB-I00 and PID2024-156032NB-I00 and the European Regional Development Fund and Diputaci\'on General de Arag\'on (E24-23R).

\bibliographystyle{elsarticle-num}
\bibliography{ref_wsb}

\begin{thebibliography}{10}
\expandafter\ifx\csname url\endcsname\relax
  \def\url#1{\texttt{#1}}\fi
\expandafter\ifx\csname urlprefix\endcsname\relax\def\urlprefix{URL }\fi
\expandafter\ifx\csname href\endcsname\relax
  \def\href#1#2{#2} \def\path#1{#1}\fi

\bibitem{Ent}
R.~L. Adler, A.~G. Konheim, M.~H. McAndrew, Topological entropy, Transactions
  of the American Mathematical Society 114 (1965) 309--319.

\bibitem{Milnor02}
J.~Milnor, Is entropy effectively computable?, Semantic Scholar (2002).

\bibitem{GHRS20}
S.~Gangloff, A.~Herrera, C.~Rojas, M.~Sablik, Computability of topological
  entropy: From general systems to transformations on cantor sets and the
  interval, Discrete and Continuous Dynamical Systems 40~(7) (2020) 4259--4286.
\newblock \href {https://doi.org/10.3934/dcds.2020180}
  {\path{doi:10.3934/dcds.2020180}}.

\bibitem{ROSSLER1976397}
O.~{R}\"ossler, An equation for continuous chaos, Physics Letters A 57~(5)
  (1976) 397--398.
\newblock \href {https://doi.org/https://doi.org/10.1016/0375-9601(76)90101-8}
  {\path{doi:https://doi.org/10.1016/0375-9601(76)90101-8}}.

\bibitem{BBS09}
R.~Barrio, F.~Blesa, S.~Serrano, Qualitative analysis of the {R}\"ossler
  equations: {B}ifurcations of limit cycles and chaotic attractors, Physica D:
  Nonlinear Phenomena 238~(13) (2009) 1087--1100.
\newblock \href {https://doi.org/https://doi.org/10.1016/j.physd.2009.03.010}
  {\path{doi:https://doi.org/10.1016/j.physd.2009.03.010}}.

\bibitem{BBS14}
R.~Barrio, F.~Blesa, S.~Serrano, Unbounded dynamics in dissipative flows:
  {R}\"ossler model, Chaos: An Interdisciplinary Journal of Nonlinear Science
  24~(2) (2014) 024407.
\newblock \href {https://doi.org/10.1063/1.4871712}
  {\path{doi:10.1063/1.4871712}}.

\bibitem{BBSS11}
R.~Barrio, F.~Blesa, S.~Serrano, A.~Shilnikov, Global organization of spiral
  structures in biparameter space of dissipative systems with {S}hilnikov
  saddle-foci, Phys. Rev. E 84 (2011) 035201.

\bibitem{BBS12}
R.~Barrio, F.~Blesa, S.~Serrano, Topological changes in periodicity hubs of
  dissipative systems, Phys. Rev. Lett. 108 (2012) 214102.
\newblock \href {https://doi.org/10.1103/PhysRevLett.108.214102}
  {\path{doi:10.1103/PhysRevLett.108.214102}}.

\bibitem{Z}
P.~Zgliczy\'nski, Computer assisted proof of chaos in the {R}\"ossler equations
  and in the {H}\'enon map, Nonlinearity 10~(1) (1997) 243--252.
\newblock \href {https://doi.org/10.1088/0951-7715/10/1/016}
  {\path{doi:10.1088/0951-7715/10/1/016}}.

\bibitem{Neumaier_1991}
A.~Neumaier, Interval Methods for Systems of Equations, Encyclopedia of
  Mathematics and its Applications, Cambridge University Press, 1991.

\bibitem{Moore}
R.~E. Moore, R.~B. Kearfott, M.~J. Cloud,
  \href{https://epubs.siam.org/doi/abs/10.1137/1.9780898717716}{Introduction to
  Interval Analysis}, Society for Industrial and Applied Mathematics, 2009.
\newblock \href
  {http://arxiv.org/abs/https://epubs.siam.org/doi/pdf/10.1137/1.9780898717716}
  {\path{arXiv:https://epubs.siam.org/doi/pdf/10.1137/1.9780898717716}}, \href
  {https://doi.org/10.1137/1.9780898717716}
  {\path{doi:10.1137/1.9780898717716}}.
\newline\urlprefix\url{https://epubs.siam.org/doi/abs/10.1137/1.9780898717716}

\bibitem{TuckerBook}
W.~Tucker, Validated Numerics: A Short Introduction to Rigorous Computations,
  Princeton University Press, USA, 2011.

\bibitem{NedialkovJackson1998}
N.~S. Nedialkov, K.~R. Jackson, An interval hermite-obreschkoff method for
  computing rigorous bounds on the solution of an initial value problem for an
  ordinary differential equation, Developments in Reliable Computing 5 (1998)
  289--310.

\bibitem{NedialkovJacksonCorliss1999}
N.~Nedialkov, K.~Jackson, G.~Corliss, Validated solutions of initial value
  problems for ordinary differential equations, Applied Mathematics and
  Computation 105~(1) (1999) 21 -- 68.
\newblock \href
  {https://doi.org/http://dx.doi.org/10.1016/S0096-3003(98)10083-8}
  {\path{doi:http://dx.doi.org/10.1016/S0096-3003(98)10083-8}}.

\bibitem{NedialkovJacksonPryce2001}
N.~S. Nedialkov, K.~R. Jackson, J.~D. Pryce, An effective high-order interval
  method for validating existence and uniqueness of the solution of an ivp for
  an ode, Reliable Computing 7~(6) (2001) 449--465.
\newblock \href {https://doi.org/10.1023/A:1014798618404}
  {\path{doi:10.1023/A:1014798618404}}.

\bibitem{Nedialkov2006}
N.~S. Nedialkov, Vnode-lp: A validated solver for initial value problems in
  ordinary differential equations, Tech. Rep. Technical Report CAS-06-06-NN
  (2006).

\bibitem{RauhBrillGunter2009}
A.~Rauh, M.~Brill, C.~G\"{u}Nther, A novel interval arithmetic approach for
  solving differential-algebraic equations with valencia-ivp, Int. J. Appl.
  Math. Comput. Sci. 19~(3) (2009) 381--397.
\newblock \href {https://doi.org/10.2478/v10006-009-0032-4}
  {\path{doi:10.2478/v10006-009-0032-4}}.

\bibitem{capd}
T.~Kapela, M.~Mrozek, D.~Wilczak, P.~Zgliczyński, {CAPD::DynSys: A flexible
  C++} toolbox for rigorous numerical analysis of dynamical systems,
  Communications in Nonlinear Science and Numerical Simulation 101 (2021)
  105578.
\newblock \href {https://doi.org/https://doi.org/10.1016/j.cnsns.2020.105578}
  {\path{doi:https://doi.org/10.1016/j.cnsns.2020.105578}}.

\bibitem{C1Lohner}
P.~Zgliczy\'nski, {$\mathcal C^1$}-{L}ohner algorithm, Foundations of
  Computational Mathematics 2 (2002) 429--465.

\bibitem{C1HO}
I.~Walawska, D.~Wilczak, An implicit algorithm for validated enclosures of the
  solutions to variational equations for {ODE}s, Applied Mathematics and
  Computation 291 (2016) 303--322.
\newblock \href {https://doi.org/https://doi.org/10.1016/j.amc.2016.07.005}
  {\path{doi:https://doi.org/10.1016/j.amc.2016.07.005}}.

\bibitem{CnLohner}
D.~Wilczak, P.~Zgliczy\'nski, {$\mathcal C^r$}-{L}ohner algorithm, Schedae
  Informaticae 20 (2011) 9--46.

\bibitem{poincare}
T.~Kapela, D.~Wilczak, P.~Zgliczyński, Recent advances in a rigorous
  computation of {P}oincar\'e maps, Communications in Nonlinear Science and
  Numerical Simulation 110 (2022) 106366.
\newblock \href {https://doi.org/https://doi.org/10.1016/j.cnsns.2022.106366}
  {\path{doi:https://doi.org/10.1016/j.cnsns.2022.106366}}.

\bibitem{CAPINSKI2023106998}
M.~J. Capiński, J.~D. {Mireles James}, W.~Tucker, D.~Wilczak, J.~B. {van den
  Berg}, Computer assisted proofs in dynamical systems, Communications in
  Nonlinear Science and Numerical Simulation 118 (2023) 106998.
\newblock \href {https://doi.org/https://doi.org/10.1016/j.cnsns.2022.106998}
  {\path{doi:https://doi.org/10.1016/j.cnsns.2022.106998}}.

\bibitem{CapinskiFleurantinJames2020}
M.~J. Capi\'{n}ski, E.~Fleurantin, J.~D.~M. James, Computer assisted proofs of
  two-dimensional attracting invariant tori for {ODE}s, Discrete and Continuous
  Dynamical Systems - A 40 (2020) 6681.
\newblock \href {https://doi.org/10.3934/dcds.2020162}
  {\path{doi:10.3934/dcds.2020162}}.

\bibitem{BARTHA2015339}
F.~A. Bartha, W.~Tucker,
  \href{http://www.sciencedirect.com/science/article/pii/S0096300315007067}{Fixed
  points of a destabilized {K}uramoto--{S}ivashinsky equation}, Applied
  Mathematics and Computation 266 (2015) 339 -- 349.
\newblock \href {https://doi.org/https://doi.org/10.1016/j.amc.2015.05.082}
  {\path{doi:https://doi.org/10.1016/j.amc.2015.05.082}}.
\newline\urlprefix\url{http://www.sciencedirect.com/science/article/pii/S0096300315007067}

\bibitem{Capinski2012}
M.~J. Capi\'nski, Computer assisted existence proofs of {L}yapunov orbits at
  {L}2 and transversal intersections of invariant manifolds in the
  {J}upiter-{S}un {PCR3BP}, SIAM J. Applied Dynamical Systems 11~(4) (2012)
  1723--1753.

\bibitem{CyrankaWanner}
J.~Cyranka, T.~Wanner, Computer-assisted proof of heteroclinic connections in
  the one-dimensional {O}hta--{K}awasaki model, SIAM Journal on Applied
  Dynamical Systems 17~(1) (2018) 694--731.
\newblock \href {https://doi.org/10.1137/17M111938X}
  {\path{doi:10.1137/17M111938X}}.

\bibitem{GalanteKaloshin}
J.~Galante, V.~Kaloshin, Destruction of invariant curves in the restricted
  circular planar three body problem using comparison of action, Duke Math. J.
  159~(2) (2011) 275–327.

\bibitem{cadiot2024rigorous}
M.~Cadiot, J.-P. Lessard, J.-C. Nave, Rigorous computation of solutions of
  semilinear pdes on unbounded domains via spectral methods, SIAM Journal on
  Applied Dynamical Systems 23~(3) (2024) 1966--2017.

\bibitem{DLT}
J.-P.~L. Gabriel William~Duchesne, A.~Takayasu, A rigorous integrator and
  global existence for higher-dimensional semilinear parabolic pdes via
  semigroup theory, Journal of Scientific Computing 62~(102) (2025).

\bibitem{BARRIO201280}
R.~Barrio, M.~Rodr\'{\i}guez, F.~Blesa, Computer-assisted proof of skeletons of
  periodic orbits, Computer Physics Communications 183~(1) (2012) 80--85.
\newblock \href {https://doi.org/https://doi.org/10.1016/j.cpc.2011.09.001}
  {\path{doi:https://doi.org/10.1016/j.cpc.2011.09.001}}.

\bibitem{WSB16}
D.~Wilczak, S.~Serrano, R.~Barrio, Coexistence and dynamical connections
  between hyperchaos and chaos in the 4d {R}\"ossler system: A
  computer-assisted proof, SIAM Journal on Applied Dynamical Systems 15~(1)
  (2016) 356--390.
\newblock \href {https://doi.org/10.1137/15M1039201}
  {\path{doi:10.1137/15M1039201}}.

\bibitem{Tucker2002}
W.~Tucker, A rigorous {ODE} solver and {S}male's 14th problem, Found. Comput.
  Math. 2~(1) (2002) 53--117.

\bibitem{GoraBoyarsky}
P.~Góra, A.~Boyarsky, Computing the topological entropy of general
  one-dimensional maps, Transactions of the American Mathematical Society
  323~(1) (1991) 39--49.

\bibitem{Milnor88}
J.~Milnor, W.~Thurston, On iterated maps of the interval, in: J.~C. Alexander
  (Ed.), Dynamical Systems, Springer Berlin Heidelberg, Berlin, Heidelberg,
  1988, pp. 465--563.

\bibitem{W85}
A.~Wolf, J.~B. Swift, H.~L. Swinney, J.~A. Vastano, Determining {L}yapunov
  exponents from a time series, Physica D: Nonlinear Phenomena 16~(3) (1985)
  285--317.

\bibitem{AUTO1}
E.~Doedel, A{UTO}: a program for the automatic bifurcation analysis of
  autonomous systems, Congr. Numer. 30 (1981) 265--284.

\bibitem{AUTO2}
E.~J. Doedel, R.~Paffenroth, A.~R. Champneys, T.~F. Fairgrieve, Y.~A.
  Kuznetsov, B.~E. Oldeman, B.~Sandstede, X.~J. Wang, Auto2000,
  http://cmvl.cs.concordia.ca/auto.

\bibitem{BW83}
J.~S. Birman, R.~Williams, Knotted periodic orbits in dynamical systems—i:
  {L}orenz's equation, Topology 22~(1) (1983) 47--82.

\bibitem{LDM95}
C.~Letellier, P.~Dutertre, B.~Maheu, Unstable periodic orbits and templates of
  the {R}\"ossler system: {T}oward a systematic topological characterization,
  Chaos: An Interdisciplinary Journal of Nonlinear Science 5~(1) (1995)
  271--282.

\bibitem{BLBD98}
G.~Boulant, M.~Lefranc, S.~Bielawski, D.~Derozier, A nonhorseshoe template in a
  chaotic laser model, International Journal of Bifurcation and Chaos 08~(05)
  (1998) 965--975.

\bibitem{UM10}
J.~Used, J.~C. Mart\'{\i}n, Multiple topological structures of chaotic
  attractors ruling the emission of a driven laser, Phys. Rev. E 82 (2010)
  016218.

\bibitem{SMB21}
S.~Serrano, M.~A. Mart\'{\i}nez, R.~Barrio, Order in chaos: Structure of
  chaotic invariant sets of square-wave neuron models, Chaos: An
  Interdisciplinary Journal of Nonlinear Science 31~(4) (2021) 043108.

\bibitem{G98}
R.~Gilmore, Topological analysis of chaotic dynamical systems, Rev. Mod. Phys.
  70 (1998) 1455--1529.
\newblock \href {https://doi.org/10.1103/RevModPhys.70.1455}
  {\path{doi:10.1103/RevModPhys.70.1455}}.

\bibitem{GL02}
R.~Gilmore, M.~Lefranc, The topology of chaos, Wiley-Interscience [John Wiley
  \& Sons], New York, 2002.

\bibitem{KG85}
H.~Kantz, P.~Grassberger, Repellers, semi-attractors, and long-lived chaotic
  transients, Physica D: Nonlinear Phenomena 17~(1) (1985) 75--86.
\newblock \href {https://doi.org/https://doi.org/10.1016/0167-2789(85)90135-6}
  {\path{doi:https://doi.org/10.1016/0167-2789(85)90135-6}}.

\bibitem{GH}
J.~Guckenheimer, P.~Holmes, Nonlinear Oscillations, Dynamical Systems, and
  Bifurcations of Vector Fields, Springer, Berlin, 1983.

\bibitem{Katok_Hasselblatt_1995}
A.~Katok, B.~Hasselblatt, Introduction to the Modern Theory of Dynamical
  Systems, Encyclopedia of Mathematics and its Applications, Cambridge
  University Press, 1995.

\bibitem{github}
D.~Wilczak, Supplementary material -- repository of the {C}++ source code,
  \texttt{https://github.com/dbwilczak/entropy-growth}.

\bibitem{WalawskaWilczak2019}
I.~Walawska, D.~Wilczak,
  \href{http://www.sciencedirect.com/science/article/pii/S1007570419300735}{Validated
  numerics for period-tupling and touch-and-go bifurcations of symmetric
  periodic orbits in reversible systems}, Communications in Nonlinear Science
  and Numerical Simulation 74 (2019) 30 -- 54.
\newblock \href {https://doi.org/https://doi.org/10.1016/j.cnsns.2019.03.005}
  {\path{doi:https://doi.org/10.1016/j.cnsns.2019.03.005}}.
\newline\urlprefix\url{http://www.sciencedirect.com/science/article/pii/S1007570419300735}

\bibitem{ZGi}
P.~Zgliczy\'nski, M.~Gidea, Covering relations for multidimensional dynamical
  systems, Journal of Differential Equations 202~(1) (2004) 32--58.
\newblock \href {https://doi.org/https://doi.org/10.1016/j.jde.2004.03.013}
  {\path{doi:https://doi.org/10.1016/j.jde.2004.03.013}}.

\bibitem{E}
R.~W. Easton, Isolating blocks and symbolic dynamics, Journal of Differential
  Equations 17~(1) (1975) 96--118.
\newblock \href {https://doi.org/https://doi.org/10.1016/0022-0396(75)90037-6}
  {\path{doi:https://doi.org/10.1016/0022-0396(75)90037-6}}.

\end{thebibliography}

\end{document}